\documentclass[12pt]{amsart}
\usepackage{amscd,amsmath,amsthm,amssymb}
\usepackage{color}
\usepackage{amsfonts,amssymb,amscd,amsmath,enumerate,verbatim,enumitem}
\usepackage{lineno}
 \def\NZQ{\mathbb}               % the font for N,Z,Q,R,C
 \def\NN{{\NZQ N}}
 
 \def\ZZ{{\NZQ Z}}
 \def\RR{{\NZQ R}}

 \def\FF{{\NZQ F}}

 \def\frk{\mathfrak}               % font for "Fraktur"
 \def\aa{{\frk a}}
 \def\pp{{\frk p}}

 \def\mm{{\frk m}}
 
 \def\nn{{\frk n}}
 \def\Jc{{\mathcal J}}
 
 \def\G{{\mathcal G}}

 \def\B{{\mathcal B}}
 
  \def\Mc{{\mathcal M}}

 \def\xb{{\mathbf x}}

 \def\opn#1#2{\def#1{\operatorname{#2}}} % to make operators
 \opn\chara{char} \opn\length{\ell} \opn\pd{pd} \opn\rk{rk}
 \opn\projdim{proj\,dim} \opn\injdim{inj\,dim} \opn\rank{rank}
 \opn\depth{depth} \opn\grade{grade} \opn\height{height}
 \opn\embdim{emb\,dim} \opn\codim{codim}
 
 \opn\Tr{Tr} \opn\bigrank{big\,rank}
 \opn\superheight{superheight}\opn\lcm{lcm}
 \opn\trdeg{tr\,deg}%\emph{
 \opn\reg{reg} \opn\lreg{lreg} \opn\ini{in} \opn\lpd{lpd}
 \opn\size{size} \opn\sdepth{sdepth}
 \opn\link{link}\opn\fdepth{fdepth}\opn\lex{lex}
 \opn\tr{tr}
 \opn\div{div} \opn\Div{Div} \opn\cl{cl} \opn\Cl{Cl}
 \opn\Spec{Spec} \opn\Supp{Supp} \opn\supp{supp} \opn\Sing{Sing}
 \opn\Ass{Ass} \opn\Min{Min}\opn\Mon{Mon}
 \opn\Ann{Ann} \opn\Rad{Rad} \opn\Soc{Soc}
 \opn\Im{Im} \opn\Ker{Ker} \opn\Coker{Coker} \opn\Am{Am}
 \opn\Hom{Hom} \opn\Tor{Tor} \opn\Ext{Ext} \opn\End{End}
 \opn\Aut{Aut} \opn\id{id}
 
 \opn\nat{nat}
 \opn\pff{pf}%   \pf exists already
 \opn\Pf{Pf} \opn\GL{GL} \opn\SL{SL} \opn\mod{mod} \opn\ord{ord}
 \opn\Gin{Gin} \opn\Hilb{Hilb}\opn\sort{sort}
 \opn\PF{PF}\opn\Ap{Ap}
 \opn\aff{aff} \opn
 \con{conv} \opn\relint{relint} \opn\st{st}
 \opn\lk{lk} \opn\cn{cn} \opn\core{core} \opn\vol{vol}  \opn\inp{inp} \opn\nilpot{nilpot}
 \opn\link{link} \opn\star{star}\opn\lex{lex}\opn\set{set}
 \opn\width{wd}
 \opn\Fr{F}
 \opn\QF{QF}
 \opn\G{G}
 \opn\type{type}\opn\res{res}
 \opn\gr{gr}
 
 \def\pot#1#2{#1[\kern-0.28ex[#2]\kern-0.28ex]}

 \opn\dirlim{\underrightarrow{\lim}}
 \opn\inivlim{\underleftarrow{\lim}}
 \let\union=\cup
 
 \let\dirsum=\oplus
 \let\tensor=\otimes
 \let\iso=\cong
 \let\Union=\bigcup
 
 \let\Dirsum=\bigoplus
 
 \let\to=\rightarrow
 \let\To=\longrightarrow
 \def\Implies{\ifmmode\Longrightarrow \else
         \unskip${}\Longrightarrow{}$\ignorespaces\fi}
 \def\implies{\ifmmode\Rightarrow \else
         \unskip${}\Rightarrow{}$\ignorespaces\fi}
 \def\iff{\ifmmode\Longleftrightarrow \else
         \unskip${}\Longleftrightarrow{}$\ignorespaces\fi}

 \let\:=\colon
 \newtheorem{Theorem}{Theorem}[section]
 \newtheorem{Lemma}[Theorem]{Lemma}
 \newtheorem{Corollary}[Theorem]{Corollary}
 \newtheorem{Proposition}[Theorem]{Proposition}
 \newtheorem{Remark}[Theorem]{Remark}
 
 \newtheorem{Example}[Theorem]{Example}
 
 \newtheorem{Definition}[Theorem]{Definition}

 \let\epsilon\varepsilon
 \let\kappa=\varkappa
 \def\qed{\ifhmode\textqed\fi
       \ifmmode\ifinner\quad\qedsymbol\else\dispqed\fi\fi}
 \def\textqed{\unskip\nobreak\penalty50
        \hskip2em\hbox{}\nobreak\hfil\qedsymbol
        \parfillskip=0pt \finalhyphendemerits=0}
 \def\dispqed{\rlap{\qquad\qedsymbol}}
 
 \opn\dis{dis}
 \def\pnt{{\raise0.5mm\hbox{\large\bf.}}}
 
 \opn\Lex{Lex}

\begin{document}
%\linenumbers
\title {The trace of the canonical module}

\author {J\"urgen Herzog, Takayuki Hibi and Dumitru I.\ Stamate}

\address{J\"urgen Herzog, Fachbereich Mathematik, Universit\"at Duisburg-Essen, Campus Essen, 45117
Essen, Germany} \email{juergen.herzog@uni-essen.de}

\address{Takayuki Hibi, Department of Pure and Applied Mathematics, Graduate School of Information Science and Technology,
Osaka University, Toyonaka, Osaka 560-0043, Japan}
\email{hibi@math.sci.osaka-u.ac.jp}

\address{Dumitru I. Stamate, ICUB/Faculty of Mathematics and Computer Science, University of Bucharest, Str. Academiei 14, Bucharest -- 010014, Romania }
\email{dumitru.stamate@fmi.unibuc.ro}

\dedicatory{ }

\begin{abstract}
The trace of the canonical module (the canonical trace)  determines the non-Gorenstein locus of a local Cohen--Macaulay ring.
We call a local Cohen--Macaulay ring nearly Gorenstein, if its canonical trace contains the maximal ideal.
Similar definitions can be made for positively graded Cohen--Macaulay $K$-algebras.
We study the canonical trace for tensor products and  Segre products of algebras,
 as well as  of (squarefree)  Veronese subalgebras.
The results are used to classify the nearly Gorenstein Hibi rings.
We study connections between the class of nearly Gorenstein rings and that of almost Gorenstein rings. We show that in dimension one, the former class includes the latter. 
\end{abstract}

\thanks{}

%		13H10   	Special types (Cohen-Macaulay, Gorenstein, Buchsbaum, etc.)
%		13D02   	Syzygies, resolutions, complexes
%		05E40   	Combinatorial aspects of commutative algebra
%%%%%		16S36   	Ordinary and skew polynomial rings and semigroup rings

%		14M25   	Toric varieties, Newton polyhedra [See also 52B20]
%		13A02   	Graded rings
%		13F20   	Polynomial rings and ideals; rings of integer-valued polynomials
%%%%%		13A18   	Valuations and their generalizations  
%		06A11   	Algebraic aspects of posets


\maketitle

\setcounter{tocdepth}{1}
\tableofcontents

\section*{Introduction}

In this paper we study the trace of the canonical module of a Cohen-Macaulay ring $R$. The Cohen-Macaulay ring may either be a local ring admitting a canonical module or else a finitely generated positively graded $K$-algebra, where $K$ is a field. All of the definitions and results which are  phrased for local rings have their  analogous correspondence for graded rings. Thus for the general facts about traces we will restrict ourselves to local rings, unless otherwise stated.

Recall that for an $R$-module $M$  one defines the trace of $M$, denoted $\tr(M)$, as the sum of the ideals $\varphi(M)$,
where the sum is taken over all $R$-module homomorphisms $\varphi\: M\to R$.
  Traces of modules have been considered in various contexts, in particular to better understand the center of the ring of endomorphisms of a module, cf. \cite{Lindo}.
Here we are only interested in the trace of the canonical module.

The significance of the trace   of the canonical module $\omega_R$ arises from the fact that it describes the non-Gorenstein locus of $R$, see Lemma~\ref{non}. In particular, $R$ is Gorenstein if $\tr(\omega_R)=R$,  and it is Gorenstein on the punctured spectrum of $R$ if $\tr(\omega_R)$ is $\mm$-primary, where $\mm$ is the maximal ideal of $R$.
 Ding \cite{Ding} has already considered $\tr(\omega_R)$. 
He showed that $\tr(\omega_R)$ is $\mm$-primary if and only if the Auslander index of $R$ is finite, see \cite[Theorem 1.1]{Ding}. 

 We call $R$ {\em nearly Gorenstein}, if $\mm\subseteq \tr(\omega_R)$. Thus $R$ is nearly Gorenstein but not Gorenstein, if and only if $\tr(\omega_R)=\mm$.
  We would like to understand how much this class differs from the Gorenstein one, and compare it with related concepts like almost Gorenstein, in the sense of \cite{BF} and \cite{GTT}.
 Using a result of Teter, in  \cite[Corollary 2.2]{HunekeVraciu} Huneke and Vraciu  showed  that quotients $R$ of Gorenstein artinian 
local rings by their socle satisfy the condition $\mm\subseteq \tr(\omega_R)$. 
They also classified the $\mm$-primary monomial ideals $I$ in a polynomial ring $S$ such that $S/I$ is of small type and it is nearly Gorenstein, see \cite[Example 4.3, Theorem 4.5]{HunekeVraciu}. 

In this paper, after proving general statements about $\tr(\omega_R)$ and how it can be computed, we apply these results to study the nearly Gorenstein property
for several classes of algebras, including the ones of Veronese type, Segre products, Hibi rings, or one dimensional  rings.

It is easily seen that if $I\subset R$ is an ideal with $\grade I>0$, then  $\tr(I)=I\cdot I^{-1}$. Here $I^{-1}$ denotes the inverse ideal of $I$, namely 
$I^{-1}=\{x\in Q(R)\: xI\subseteq R\}$ with $Q(R)$  the total ring of fractions  of $R$. In particular, if $R$ is generically Gorenstein, but not Gorenstein, $\omega_R$ may be identified with an ideal of grade $1$, and  hence in this case  $\tr(\omega_R)=\omega_R\omega_R^{-1}$,  where $\omega_R^{-1}$ is the anti-canonical ideal of $R$. While the canonical ideal $\omega_R$ is only determined up to isomorphism, its  trace is uniquely determined.

For the convenience of the reader we collect and record in Section~\ref{basic} a few general and basic properties on the trace of modules, some of them well-known. In Proposition~\ref{prop:ideals-trace} the trace of a product of ideals is considered in relation to the traces of its factors, while  Lemma~\ref{lemma:trace-changeofrings} describes the behavior of the trace under change of rings.

In  Section~\ref{sec:nearly}  nearly Gorenstein rings are introduced and it is shown in Proposition \ref{behaviour} that if a ring  is nearly Gorenstein, then any reduction of the ring modulo a regular sequence  is again nearly Gorenstein. It is also observed that the converse  does not hold in general.

In the case that $R$ is a residue class ring of a regular local ring $S$, $\tr(\omega_R)$ can be computed in terms of the free $S$-resolution of $R$, more precisely in terms of the last step  of the resolution, as shown in Corollary~\ref{first}. This  is the consequence  of a result of Vasconcelos \cite{Vasconcelos} which we recall in Proposition~\ref{vasconcelos}. In the rest of Section~\ref{sec:computation-trace} other, more special,  situations are considered in which $\tr(\omega_R)$ can be computed more explicitly. Notable is the situation when $R$ is generically Gorenstein, e.g. a domain, and the Cohen--Macaulay type of $R$ is $2$. It is shown in Corollary~\ref{third} that in this case the entries of the last map in the free $S$-resolution of $R=S/I$ generate $\tr(\omega_R)$  modulo $I$.

Section~\ref{sec:special} is devoted to the study of the canonical trace ideal for special algebra constructions. There we have in mind tensor products of algebras, Veronese algebras, squarefree Veronese algebras and Segre products. The result for tensor algebras (Theorem~\ref{thm:canonical-tensor-product}) asserts the following: let $K$ be a field and $R_1$ and $R_2$ be positively graded  Cohen-Macaulay $K$-algebras, and let $\omega_{R_1}$, $\omega_{R_2}$ be their respective canonical modules. Then $\tr_R(\omega_{R_1\tensor_KR_2}) = \tr_{R_1}(\omega_{R_1})R \cdot \tr_{R_2}(\omega_{R_2})R$.  As  an  immediate consequence one obtains  that the tensor product $R_1\tensor_KR_2$ is nearly Gorenstein if and only if it    is Gorenstein,  which in turn is the case if and only if both $R_1$ and $R_2$ are Gorenstein, see Corollary~\ref{cor:tensor-NG}. In particular,   a polynomial ring extension of a positively graded $K$-algebra $R$ is nearly Gorenstein if and only if $R$ is Gorenstein, see Corollary~\ref{polynomial}. A similar statement holds for power series over a local ring, as stated in Proposition~\ref{localanalogue}.

Next we consider Veronese subalgebras of a standard graded $K$-algebra $R$. Let $d>0$ be an integer and  $R^{(d)}=\Dirsum_iR_{id}$ be the $d$-th Veronese subalgebra of $R$.
The $R^{(d)}$-modules $M_j=\Dirsum_{i\in \NN}R_{di+j}$ for $j=0,\dots, d-1$ are called the $d$-th Veronese submodules of $R$. In Theorem~\ref{thm:vero-modules} we show that if $R$ has positive depth, then $\mm^{(d)} \subseteq \tr_{R^{(d)}}(M_j)$ for any $d$-th Veronese submodule $M_j$ of $R$. Here $\mm^{(d)}$ denotes the graded maximal ideal of $R^{(d)}$. By using a result of Goto and Watanabe \cite{GW} one deduces then from Theorem~\ref{thm:vero-modules} that all Veronese subalgebras of a standard graded $K$-algebra $R$ over an infinite field $K$ are nearly Gorenstein, if $R$ is Gorenstein.

The situation for squarefree Veronese subalgebras is  more complicated.
Given integers $n\geq d>0$, the  $d$-th squarefree Veronese subalgebra $R_{n,d}$ of the polynomial ring $S$ in $n$ variables over a field $K$  is the $K$-algebra
generated by the squarefree monomials in $S$ of degree $d$.
Based on a theorem of Bruns, Vasconcelos and Villarreal  \cite{BVV} we give in Theorem~\ref{thm:sqfreevero-anticanonical} an explicit description of the anti-canonical ideal of
$R_{n,d}$, and we use this result to show in Theorem~\ref{thm:sqfreevero}   that the following conditions are equivalent:
(i) $R_{n,d}$ is nearly Gorenstein,  (ii) $R_{n,d}$ is Gorenstein,  (iii) $d=1$ or $d=n-1$ or $n=2d$.

By another  result of Goto and Watanabe \cite{GW}, the canonical module for the Segre product  $T=R\# S$ of positively graded Cohen-Macaulay $K$-algebras $R$ and $S$ of Krull dimension at least $2$
 is just the Segre product of the respective canonical modules,  assuming that $T$ is Cohen--Macaulay.
We use this result in Theorem \ref{general segre} to compute $\tr(\omega_T)$ in the case that $R$ and $S$ are standard graded Gorenstein $K$-algebras.
It is shown in Theorem~\ref{general segre} that $\mm_T^{|r-s|}\subseteq \tr(\omega_T)$,
where $\mm_T$ is the graded maximal ideal of $T$ and where $r$ and $s$ are the respective $a$-invariants of $R$ and $S$.
Equality holds, if    $R$ and $S$ are homogeneous semigroup rings .
It follows that under these conditions $T$ is nearly Gorenstein if and only if $|r-s|\leq 1$, see Corollary~\ref{nearlysegre}. The section ends with Proposition~\ref{lemma:i} in which the  anti-canonical ideal of the Segre product is computed in the case that $R$ and $S$ are polynomial rings.

These results on Segre products are used in Section~\ref{sec:hibi} to give a complete classification of all nearly Gorenstein Hibi rings. Given a finite distributive lattice $L$ and a field $K$, the Hibi ring of $L$ defined over $K$ is the toric ring $\mathcal{R}_K[L]$ whose relations are the meet-join relations of $L$. By a fundamental theorem of Birkhoff, $L$ is the ideal lattice $\Jc(P)$ of its poset of join irreducible elements. It is shown in Theorem~\ref{nearlygorhibiring} that  $\mathcal{R}_K[L]$ is nearly Gorenstein
if and only if $P$ is the disjoint union of pure connected posets
$P_1, \ldots, P_q$ such that $|\rank(P_i) - \rank(P_j)| \leq 1$
for $1 \leq i < j \leq q$.
 This naturally complements the result of the second author in \cite{TH87} that $\mathcal{R}_K[L]$ is Gorenstein if and only if $P$ is pure.

In Section~\ref{sec:onedimensional} we consider one-dimensional local Cohen--Macaulay rings admitting a  canonical module. 
Our first observation is that any one-dimensional almost Gorenstein ring as defined by Barucci and Fr\"oberg \cite{BF} is nearly Gorenstein, see Proposition~\ref{implication}. For this proof we use the description of almost Gorenstein rings as given by Goto et al. \cite{GTT}.
If the one-dimensional Cohen--Macaulay ring  $R$ is of embedding dimension 3 and 
it has the presentation $S/I$ with $S$ a regular local ring of dimension $3$ with $\mu(I)=3$, 
then $R$ is nearly Gorenstein if and only if the entries of the relation matrix of $I$ generate the maximal ideal of $S$, see  Proposition \ref{codim2}.

Unfortunately, in higher dimensions almost Gorenstein rings and nearly Gorenstein rings are not related to each other, as it is shown by examples. However, as one of the reviewers pointed out,  if the Cohen-Macaulay local ring $R$ of minimal multiplicity  admits a canonical module, $R$ has positive dimension and an infinite residue field, then  $R$ is almost Gorenstein once it is nearly Gorenstein. This is shown in  Theorem \ref{thm:minimal}.

\medskip

\section{Basic properties of the trace}
\label{basic}

For an $R$-module $M$, its {\em trace}, denoted $\tr_R(M)$, is the sum of the ideals $\varphi(M)$ with $\varphi \in \Hom_R(M,R)$. Thus,
$$
\tr_R(M)=\sum_{\varphi \in \Hom_R(M,R)}\varphi(M).
$$
When there is no risk of confusion about the ring we simply write $\tr(M)$.

Note that if $M$ is finitely generated, the trace localizes. In other words, $\tr(M)R_P=\tr(M_P)$ for all $P\in \Spec(R)$.

\medskip
If $M_1$ and $M_2$ are isomorphic $R$-modules, then $\tr_R(M_1)=\tr_R(M_2)$.

\medskip

Given any ideal $I\subset R$ of positive grade, we set
$$
I^{-1}=\{x\in Q(R)\:\ xI\subseteq R\},$$
where  $Q(R)$ is the total ring of fractions of $R$.

\begin{Lemma}
\label{lemma:traceideal}
Let $I\subset R$ be an ideal of positive grade. Then $\tr(I)=I\cdot I^{-1}$.
\end{Lemma}

\begin{proof}
Pick  $b \in I$ and which is regular on $R$. For any $a\in I$ and any $\varphi \in \Hom_R(I, R)$ one has $\varphi(ab)=b \varphi(a)=a\varphi(b)$,
hence $\varphi(a)= a \cdot (\varphi(b)/b)$. We claim that $\varphi(b)/b \in I^{-1}$. Indeed, for any $c\in I$ one has $(\varphi(b)/b)c= (\varphi(b)c)/b=\varphi(bc)/b=(b\varphi(c))/b=\varphi(c) \in R$.
This shows that $\tr(I) \subseteq I\cdot I^{-1}$.

For the reverse inclusion, note that $I \cdot I^{-1}=\sum_{x\in I^{-1}}xI$.  Since $xI$ is the image of the $R$-linear map  $\theta_x:I \to R$ with $\theta_x(a)=xa$, the assertion follows.
\end{proof}

\begin{Remark}
{\em
According to \cite[Exercise 1.2.24]{BH}, for any ideal $I$ of the noetherian ring $R$,  $\grade I \geq 2$ if and only if
the canonical homomorphism $R \to \Hom_R(I,R)$ is an isomorphism. Equivalently, any $R$-module homomorphism  from $I$ into $R$ is just the multiplication by some element in $R$.
Therefore, if $\grade I \geq 2$ one has $\tr(I)=I$.
}
\end{Remark}
We record some properties of the trace ideals that we need for later.

\begin{Proposition}
\label{prop:tensor-trace}
Let $M$ and $N$ be two $R$-modules. Then
$$
\tr(M) \tr(N) \subseteq \tr(M \otimes_R N) \subseteq \tr(M) \cap \tr(N).
$$
\end{Proposition}

\begin{proof}
Let $a\in M$, $b\in N$, $\varphi \in \Hom_R(M,R)$ and $\psi \in \Hom_R(N, R)$.
If we denote $\epsilon : R\otimes_R R \to R$ the canonical map letting $\epsilon (u\otimes v)=uv$ for any $u,v$ in $R$, then
$\varphi(a)\psi(b) = (\epsilon \circ  (\varphi \otimes \psi)) (a\otimes b) \in \tr(M\otimes_R N)$. This shows that $\tr(M) \tr(N) \subseteq \tr(M \otimes_R N)$.

Let $a \in M$, $b\in N$ and $\varphi \in \Hom_R(M\otimes _R N, R)$. The map $\psi :N \to R$ given by $\psi(n)= \varphi(a\otimes n)$
is $R$-linear and $\varphi(a\otimes b) = \psi(b) \in \tr(N)$.
From here we get  that $\tr(M \otimes_R N)\subseteq \tr(N)$, and by symmetry, also that $\tr(M \otimes _R N)\subseteq \tr(M) \cap \tr(N)$.
\end{proof}

\begin{Proposition}
\label{prop:ideals-trace}
Let $I$ and $J$ be ideals of positive grade in the ring $R$. Then
$$
\tr(I) \tr(J) \subseteq \tr(IJ)  \subseteq \tr(I) \cap \tr(J).
$$
\end{Proposition}

\begin{proof}
Notice first that $IJ$ has positive grade and   Lemma \ref{lemma:traceideal} applies to it, as well.

Let $a \in I$, $b\in I^{-1}$, $a'\in J$ and $b' \in J^{-1}$. Clearly $ (ab)(a'b')=(aa')(bb')$ and $aa' \in IJ$. We claim that $bb'\in (IJ)^{-1}$.
For that, it is enough to consider $u\in I$, $v\in J$ and show that $(bb')(uv) \in (IJ)^{-1}$. This is the case, since $(bb')(uv)=(bu)(b'v) \in R$.
We conclude that $\tr(I) \tr(J) \subseteq \tr(IJ)$.

For the other inclusion in the text we consider  $x=az $ with $a\in IJ$ and $z\in (IJ)^{-1}$.
One has $a=\sum_\lambda u_\lambda v_\lambda$ with $u_\lambda \in I$ and $v_\lambda \in J$ for all $\lambda$.
Clearly, $v_\lambda z \in I^{-1}$, since for any $b\in I$ one has $b(v_\lambda z)= (b v_\lambda)z \in (IJ)(IJ)^{-1} \subseteq R$.  Therefore, $u_\lambda (v_\lambda z) \in I\cdot I^{-1}$,
and also $az \in I\cdot I^{-1}$. Conclusion follows by Lemma  \ref{lemma:traceideal}.
\end{proof}

Next we  study   how the trace behaves under a base change.

\begin{Lemma}
\label{lemma:trace-changeofrings}
Let $\varphi: R_1 \to R$ be a ring homomorphism, $M$ an $R$-module, and $M_1$ an $R_1$-module. Then
\begin{enumerate}
\item [{\em (i)}] if $\varphi$ is surjective, one has $(\tr_{R_1}M)R \subseteq \tr_R M$;
\item [{\em (ii)}] $(\tr_{R_1}M_1)R \subseteq \tr_R(M_1\otimes_{R_1} R)$;
\item [{\em (iii)}] if $\varphi$ is a flat morphism, $R_1$ is a Noetherian ring and $M_1$ is a finitely generated $R_1$-module, one has  $$(\tr_{R_1}M_1)R = \tr_R(M_1\otimes_{R_1} R).$$
%% in fact, it is enough to require that $M_1$ has a finite presentation as an $R_1$-module
\end{enumerate}
\end{Lemma}

\begin{proof}
For (i) let us consider an  $R_1$-linear map  $f:M \to R_1$. We claim that the map $g=\varphi \circ f: M \to R$ is $R$-linear.
It is clearly additive. For  $r\in R$ we pick $r_1\in R_1$ with $\varphi(r_1)=r$. Hence $g(rm)=g(r_1\cdot m)=\varphi(f(r_1 \cdot m))=\varphi(r_1 f(x))=\varphi(r_1) \varphi(f(m))=r g(m)$, for all $m$ in $M$.
This shows that $\varphi(\Im f) \subseteq \Im g$, hence $(\tr_{R_1}M)R \subseteq \tr_R M$.

For (ii),  let $f:M_1 \to R_1$ be an $R_1$-linear map. Denoting $g=f\otimes_{R_1} 1_R:M_1 \otimes _{R_1} R \to R_1\otimes_{R_1}R \cong R$ the $R$-linear map given by $g(m\otimes r)=rf(m)$,
we have that $\Im(f)R \subseteq \Im (g)$, hence $(\tr_{R_1}M_1)R \subseteq \tr_R(M_1\otimes_{R_1} R)$.

Under the assumptions of (iii), by \cite[3.E]{Matsumura} there is an isomorphism
$$
\Hom_R(M_1\otimes_{R_1} R, R) \cong \Hom_{R_1}(M_1, R_1) \otimes_{R_1} R.
$$
This implies that the image of any $g\in \Hom_R(M_1\otimes_{R_1} R, R)$ is obtained by extending into $R$ the image of a suitable $f\in \Hom_{R_1}(M_1, R_1)$,
by the method described in the proof of part (ii). This gives that $\tr_R(M_1\otimes_{R_1} R) \subseteq (\tr_{R_1}M_1)R$, and using part (ii), the conclusion follows.
\end{proof}

\begin{Remark}
{\em
The inclusion in part (i) of Lemma \ref{lemma:trace-changeofrings} may be strict. Let $R$ be any ring and $I$ an  ideal of positive grade.
Then $\tr_{R/I}(R/I)=R/I$ and   $\tr_{R}(R/I)=0$, since $\Hom_R(R/I, R) \cong 0:_R I= 0$. 
%Indeed, if $x \in I$ is a regular on $R$ and $\varphi \in \Hom_R(R/I, R)$, then
% $0=\varphi(\hat{x})=\varphi(x\cdot \hat{1})=x\varphi(\hat{1})$ in $R$, hence $\varphi(\hat{1})=0$. This implies %$\varphi(\hat{r})=r\varphi(\hat{1})=0$ for any $r$ in $R$.
}
\end{Remark}

\medskip

\section{Nearly Gorenstein rings}
\label{sec:nearly}

 Let $(R,\mm)$  be a Cohen--Macaulay local ring which  admits a canonical module $\omega_R$.
The trace of $\omega_R$ describes the non-Gorenstein locus of $R$. Indeed, one has

\begin{Lemma}
\label{non}
Let $P\in \Spec(R)$. Then $R_P$ is not a Gorenstein ring if and only if $$\tr(\omega_R)\subseteq P.$$
\end{Lemma}

\begin{proof}
Let  $\tr(\omega_R)\subseteq P$. Suppose that $R_P$ is Gorenstein.
Then $\omega_{R_P}\iso R_P$, and hence there exists $\varphi\in \Hom_{R_P}(\omega_{R_P}, R_P)$ with $\varphi(\omega_{R_P})=R_P$.
It follows that $\tr(\omega_R)_P=\tr(\omega_{R_P})=R_P$. Thus $\tr(\omega_R)\not\subseteq P$.

Conversely, suppose that $\tr(\omega_R)\not\subseteq P$. Then $\tr(\omega_{R_P})=R_P$. Therefore, there exists a surjective $R_P$-module homomorphism $\varphi: \omega_{R_P}\to R_P$. Since $R_P$ is free, the map $\varphi$ splits, and we get $\omega_{R_P}\iso R_P\dirsum U$.
Since $\omega_{R_P}$ is a maximal Cohen--Macaulay module of type 1, it follows that  $\omega_{R_P}\iso R_P$, and this implies that $R_P$ is Gorenstein.
\end{proof}

If $R$ is Gorenstein on the punctured spectrum $\Spec(R)\setminus\{\mm \}$ we define the {\em residue} of $R$ as the numerical invariant
\begin{equation*}
\res(R)= \ell_R (R/\tr(\omega_R)),
\end{equation*}
where $\ell_R(-)$ denotes the length as an $R$-module.
Ananthnarayan \cite[Corollary 3.8]{Ananth} shows that if moreover $R$ is artinian, then $\res(R)$ is bounded above by the Gorenstein colength of $R$, 
  introduced by him in \cite[Definition 1.2]{Ananth}.

\begin{Definition}
{\em
A   Cohen-Macaulay  local ring  (or positively graded $K$-algebra) $R$  admitting a canonical module $\omega_R$  is called {\em nearly Gorenstein},
if $\tr(\omega_R)$ contains the (graded) maximal ideal  $\mm$ of $R$.
}
\end{Definition}

It follows from this definition that  Gorenstein rings are nearly Gorenstein, and that $R$ is nearly Gorenstein but not Gorenstein, if and only if $\tr(\omega_R)=\mm$.

\medskip
 We recall that the ring $R$ is called generically Gorenstein if $R_P$ is Gorenstein for all minimal prime ideals $P$ of $R$.
Since any field is Gorenstein, we get that domains are generically Gorenstein.
If  we assume that $R$ is generically Gorenstein, then  by \cite[Proposition 3.3.18]{BH},  $\omega_R$ is either   isomorphic to  the whole ring $R$ if the latter is Gorenstein,
or otherwise it can be identified with an ideal of $R$ of grade $1$.
It follows from  Lemma \ref{lemma:traceideal} that $\tr(\omega_R)=\omega_R \cdot \omega_R^{-1}$,
where $\omega_R^{-1}$ is   also called  the anti-canonical ideal of $R$.

\begin{Proposition}
\label{behaviour}
{\em (a)} Let $R$ be nearly Gorenstein. Then $R$ is Gorenstein on the punctured spectrum of $R$.

{\em (b)} Let $\xb=x_1,\ldots,x_r$ be a regular sequence on $R$, and set $\bar{R}=R/(\xb)$. If $R$ is nearly Gorenstein, then so is $\bar{R}$.
The converse does not hold in general.
\end{Proposition}

\begin{proof} (a) is an immediate consequence of Lemma~\ref{non}

(b) Observing that $\omega_{\bar{R}}=\omega_R\tensor_R \bar{R}$ by \cite[Theorem 3.3.5]{BH}, Lemma~\ref{lemma:trace-changeofrings} implies that
$\tr_R(\omega_{R})\bar{R}\subseteq \tr_{\bar{R}}(\omega_{\bar{R}})$. This shows that if $R$ is nearly Gorenstein, then so is $\bar{R}$.

On the other hand, assume that $\bar{R}$ is nearly Gorenstein but not Gorenstein, and let $R=\bar{R}[|x|]$ be the formal power series over $\bar{R}$. Then $x$ is a non-zerodivisor of $R$,  and $R/(x)R=\bar{R}$ is nearly Gorenstein. But $R$ itself is not nearly Gorenstein, see Proposition~\ref{localanalogue}.
\end{proof}

\medskip

\section{Computing  the trace of the canonical module}
\label{sec:computation-trace}

Let $\alpha\: R^p\to R^q$ be an $R$-linear map, and let $A$ be the matrix representing $\alpha$  with respect to some bases of $R^p$ and $R^q$. We denote by $I_t(A)$ the ideal of $t$-minors of $A$. This ideal  depends only on $\alpha$ and not on the chosen  bases. Therefore, we also write $I_t(\alpha)$ for $I_t(A)$.
Following Vasconcelos  \cite[Remark 3.3]{Vasconcelos}, the trace of a module can be computed as follows.

\begin{Proposition}
\label{vasconcelos}
Let $  F_1    \stackrel{\varphi}{\rightarrow}  F_0\to M$ be a  free presentation of  the $R$-module  $M$.
Let $\varphi^*$ denote the dual of $\varphi$, and consider the beginning of a   free resolution
\[
\begin{CD}
G@>\alpha >> F_0^* @> \varphi^* >> F_1^* @>>> D(M) @>>> 0.
\end{CD}
\]
of the Auslander dual $D(M)$  of $M$ (which is defined to be the cokernel of $\varphi^*$).  Then $\tr(M)=I_1(\alpha)$.
\end{Proposition}

\begin{proof}
 The $R$-module homomorphisms $\gamma \: M\to R$ are induced by $R$-module homomorphisms $\beta\: F_0\to R$ with $\beta\circ \varphi=0$. Thus $\tr(M)=\sum \beta(F_0)$,  where the sum is taken over all $R$-module homomorphisms $\beta\: F_0\to R$ with $\beta\circ \varphi=0$. Let $\beta$ be such an $R$-module homomorphism, and let $\B\: e_1,\ldots,e_m$ be a basis of $F_0$. Then $\beta(F_0)$ is the ideal $(\beta(e_1),\ldots,\beta(e_m))\subseteq R$.
Notice that $\beta^*(1)=\sum_{i=1}^m\beta(e_i)e_i^*$, where $e_1^*,\ldots,e_m^*$ is the dual basis of $\B$. Since $\varphi^*\circ \beta^* =(\beta\circ \varphi)^*=0$, it follows that $\sum_{i=1}^m\beta(e_i)e_i^*$ is in the image of $\alpha$, and this shows that $\tr(M)\subseteq I_1(\alpha)$.

Conversely,  let $g_1,\ldots,g_r$ be a basis of $G$ and let $\alpha(g_k)=\sum_{j=1}^ma_{jk}e_j^*$ for $k=1,\ldots,r$.
Then $I_1(\alpha)$ is the sum of the ideals  $ J_k$  with   $J_k=(a_{1k},\ldots,a_{mk})$ for $k=1,\ldots,r$.
For such $k$, let $\beta_k:\ F_0\to R$ be the $R$-module homomorphism defined by $\beta_k(e_j)=a_{jk}$ for $j=1,\ldots,m$.
Since $\sum_{j=1}^ma_{jk}e_j^*$ is in the kernel of $\varphi^*$, it follows that $\beta_k\circ \varphi=0$.
Therefore, $J_k\subseteq \tr(M)$ for all $k$.  This shows that $I_1(\alpha)\subseteq \tr(M)$.
\end{proof}

\begin{Corollary}
\label{first}
Let $(S, \nn)$ be a regular local ring and let
\[
\begin{CD}
\FF\: 0\To F_p@> \varphi_p >>  F_{p-1}\To  \cdots \To  F_1\To F_0 \To R\To 0
\end{CD}
\]
be a minimal free $S$-resolution of   the Cohen-Macaulay ring $R=S/J$  with $J\subseteq \nn^2$. 
Let $e_1,\ldots,e_t$ be a  basis of $F_p$. Suppose that  for $i=1,\ldots,s$  the elements $\sum_{j=1}^tr_{ij}e_j$ generate the
kernel of
\[
\begin{CD}
F_p\tensor R @> \psi_p >>  F_{p-1}\tensor R,
\end{CD}
\]
where $\psi_p =\varphi_p \tensor R$. Then $\tr(\omega_R)$ is generated by  the elements $r_{ij}$ with  $i=1,\ldots, s$ and $j=1,\ldots,t$.
\end{Corollary}

\begin{proof}
The canonical module $\omega_R$ of $R$  can be computed as the cokernel of the map $\varphi^*_p\: F_{p-1}^*\to  F_{p}^*$, see \cite[Corollary 3.3.9]{BH}.
Hence, as an $R$-module its presentation is given as follows:
\[
\begin{CD}
 F_{p-1}^*\tensor R@> \varphi_p^*\tensor R >>  F_{p}^*\tensor R \To \omega_R\To 0.
\end{CD}
\]
Thus the desired conclusion follows from Proposition~\ref{vasconcelos}.
\end{proof}

Recall that  the (Cohen--Macaulay) type  of a Cohen--Macaulay  local  ring $(R, \mm )$ of dimension $d$ is the  number $\dim_{R/\mm}\Ext_R^d(R/\mm, R)$.

In the following corollaries we refer to the notation of Corollary~\ref{first}.

\begin{Corollary}
\label{second}
Let $R$  be    Cohen--Macaulay   of type $t$, and assume that $R$ is generically Gorenstein. Then  $I_{t-1}(\psi_p)\subseteq \tr(\omega_R)$.
\end{Corollary}

\begin{proof}
Since $\varphi_p^*\tensor R=\psi_p^*$, it follows that $I_{t-1}(\psi_p)=I_{t-1}(\psi_p^*)=I_{t-1}(\varphi_p^*\tensor R)$.

Let $U=\Im(\varphi_p^*\tensor R)$.  Then we obtain an exact sequence $0\to U\to R^t\to \omega_R\to 0$.
Here we identified $F_{p}^*\tensor R$ with $R^t$, since $F_p$ is free of a rank $t$.

Since $R$ is generically Gorenstein, $\omega_R$ is an ideal of rank 1, so that $U$ is a module  of rank $t-1$. %\cite[Proposition 1.4.5]{BH}
Let $A$ be the
$t\times r$-matrix  which describes $U$ as a submodule of $R^t$. Then  $A$  is just the relation matrix of $\omega_R$ and  $\rank A=t-1$. %\cite[Definition 1.4.2]{BH}
This implies (\cite[Proposition 1.4.11]{BH}) that all $t$-minors of $A$ vanish.

We need to prove that $I_{t-1}(A) \subseteq \tr(\omega_R)$.

Let $\Delta$ be any $(t-1)$-minor of $A$, and let $B$ be the $t \times (t-1)$ submatrix of $A$ whose columns are involved in computing $\Delta$.
For any $j=1,\dots, r$, adding the column  $(a_{1j}, \dots, a_{tj})^T$ of $A$ to $B$, we obtain a $t\times t$-matrix $B'$ whose determinant  is zero, since $I_t(A)=0$.
Expanding $B'$ with respect to the new column we have added to $B$, 
we see that $\sum_{i=1}^t(-1)^{i}a_{ij}\Delta_i=0$,  where $\Delta_i$ is the determinant of the $(t-1)\times (t-1)$-matrix which is obtained from $B$ by dropping the $i$th row.

In conclusion, we see that $U$ is in the kernel of the map $\alpha:R^t\to R$ which assigns to the $i$th canonical basis vector of $R^t$
the element $(-1)^i\Delta_i$. 
This shows that $(\Delta_1, \dots, \Delta_t)R\subseteq \Im (\bar{\alpha})$, where $\bar\alpha:R^t/U \to R$ is the $R$-linear map   induced by $\alpha$.  Hence $\Delta_i \in \tr(\omega_R)$ for any $i=1,\dots, t$. 
In particular, $\Delta \in \tr(\omega_R)$. This finishes the proof.
\end{proof}

\begin{Corollary}
\label{third}
Suppose $R$ is a  generically Gorenstein ring  of type $2$. Then 
$$\tr(\omega_R)=I_1(\psi_p)   \text{ and } \omega_R^{-1} \cong \Im(\psi^*_p). 
$$
In particular, $\mu(\omega_R^{-1}) \leq \rank F_{p-1}$.
\end{Corollary}
%% Nice observation: the rank(\omega_R^{-1})= rank U= 1, as seen in the proof of the corollary above. 
%% However, \omega_R^{-1} is minimally generated over R by _at most_ \rank F_{p-1} relations.

\begin{proof}
Regarding the first equation, we only need to show that $\tr(\omega_R)\subseteq I_1(A)$, 
since $I_1(\psi_p)=I_1(A)$ and since the other inclusion is already shown in Corollary~\ref{second}.

 As $R$ is generically Gorenstein, but not a Gorenstein ring, by \cite[Proposition 3.3.18]{BH} we may assume that $\omega_R$ is an ideal in $R$ of height one.
It follows that $\tr(\omega_R)=\omega_R \cdot \omega_R^{-1}$. Let $f_1$ and $f_2$ be the generators of $\omega_R$.
We may assume that $A$ is the relation matrix of $\omega_R$ with respect to these generators.
Then  it suffices to show that the elements $xf_1$ and $xf_2$ belong to $I_1(A)$ for any $x\in \omega_R^{-1}$.
To see this, let $a=xf_1$ and $b=xf_2$. 
%  Then $x(b f_1- a f_2)=0$. Since  $R$ is a domain, we conclude that
Since  $bf_1-af_2=0$  we get that 
 % $(b,-a)$  is a relation between $f_1$ and $f_2$, hence
 the vector $(b,-a)$ is in the $R$-module spanned by the columns of the  matrix $A$. This implies that  $a,b\in I_1(\psi_p)$.

Letting $r=\rank F_{p-1}$,
$$
\begin{CD}
R^r  @> \psi_p^* >>R^2 \To \omega_R \To 0
\end{CD}
$$
is the beginning of  a (not necessarily minimal) free resolution for the ideal $\omega_R=(f_1, f_2)$. 
Then $U=\Im (\psi_p^*) =\{ (a,b) \in R^2: af_1+bf_2=0\}$.
It is an easy exercise to check that the map $\omega_R^{-1}\to U$ with $g\mapsto(gf_2, -gf_1)$ establishes an isomorphism, 
%surjectivity is clear. For the injectivity note that $\omega_R$ is faithful $R$-module, hence $\ann_R(\omega_R)=0$. Therefore, if $gf_1=gf_2=0$ then $g=0$.
hence also $\mu(\omega_R^{-1})=\mu(U) \leq \rank F_{p-1}$.
\end{proof}

By using Corollary~\ref{third} we obtain the following characterization of  nearly Gorenstein rings of type $2$ and  having  positive dimension.

\begin{Corollary}
\label{typetwo}
Let $R$ be  a   ring  of type $2$ with $\dim R >0$. Then $R$ is nearly Gorenstein if and only if $I_1(\varphi_p)=\nn$, where $\nn$ is  the maximal ideal of $S$.
\end{Corollary}

\begin{proof}
 If $R$ is nearly Gorenstein, then it is Gorenstein on the punctured spectrum (see Proposition \ref{behaviour}(a)), which implies that $R$ is generically Gorenstein. Corollary~\ref{third} gives that $I_1(\psi_p)=\mm$, hence  $I_1(\varphi_p)+J=\nn+J=\nn $.
Since $J\subseteq \nn^2$ we get that $I_1(\varphi_p)+ \nn^2 =\nn$, and by Nakayama's lemma it follows that $I_1(\varphi_p)=\nn$.

Conversely, let us assume that  $I_1(\varphi_p)=\nn$. 
Then $I_1(\psi_p)=\mm$ and it follows from Corollary \ref{third} that  $R$ is nearly Gorenstein once we show that $R$ is generically Gorenstein. Let $\pp$ be a minimal prime ideal over $J$.  
As $(I_1(\psi_p))_\pp= R_\pp$, the last step in the resolution $\FF_\pp$ is not minimal.  
Since $R_\pp$ is Cohen-Macaulay and $\height J_\pp=\height J$, it follows that $\projdim_{S_\pp}R_\pp=\projdim_S R=p$. 
Therefore, the type of $R_\pp$, which is the $p$-th Betti number in a minimal free resolution of $R_\pp$ over $S_\pp$, is strictly less than the rank of $F_p$, which equals $2$. Hence $R_\pp$ is a Gorenstein ring.
%% If $I_1(\varphi_p)=\nn$,  clearly $R$ is nearly Gorenstein. Now assume that $R$ is nearly Gorenstein.
%% If $J\subseteq \nn^2$, we get that $I_1(\varphi_p)+J=\nn \subseteq I_1(\varphi_p)+\nn^2 \subseteq \nn$, hence
%% $I_1(\varphi_p)+\nn^2=\nn$. Since $I_1(\varphi_p)\subseteq \nn$, by Nakayama's lemma we obtain the desired result.
%% If $J$ contains an element $x \in \nn\setminus \nn^2$, this is part of a minimal system of generators for $J$ and $\nn$ and
%% it is regular on $S$. Thus the minimal resolution of $S/J$ over $S$ is obtained from the one over $S/(x)$ tensored with $K[x]/(x)$
%% This way we reduce to the case when $J\setminus \nn^2$.
\end{proof}

\medskip

\section{The canonical trace ideal for special constructions of algebras}
\label{sec:special}

\subsection{Tensor products}

In this section we describe the trace of a tensor product of two $K$-algebras as the product of the trace ideals of the factors.

\begin{Proposition}
\label{prop:zero-tor}
Let $R_1$ and $R_2$ be  $K$-algebras over the field $K$, and set $R=R_1\otimes_K R_2$.
For any finitely presented %\footnote{we need this finiteness condition only  in part (ii) because we invoke Lemma \ref{lemma:trace-changeofrings} }
  $R_1$-module $M_1$ and $R_2$-module $M_2$ one has that
\begin{enumerate}
\item[{\em (i)}] $\Tor_i^R(M_1\otimes_{R_1}R, M_2\otimes_{R_2}R)=0$ for all $i>0$;
\item[{\em (ii)}] $\tr_{R_1}(M_1)R \cdot \tr_{R_2} (M_2)R = \tr_R(M_1\otimes_K M_2) =\tr_{R_1} (M_1)R \cap \tr_{R_2} (M_2)R$.
\end{enumerate}
\end{Proposition}

\begin{proof}
Let   $\mathbb{F}'\to 0$ and  $\mathbb{G}' \to 0$ be free resolutions of $M_1$ over $R_1$, and of $M_2$ over $R_2$, respectively.
General homological algebra facts (see \cite[Theorem 2.7.2]{Weibel})
 imply that for all $i>0$ one has $H_i(\mathbb{F}'\otimes_K \mathbb{G}') \cong \Tor_i^K(M_1, M_2)=0$, since $M_1$ is free over $K$.

Since the canonical  maps $R_1 \to R$ and $R_2\to R$ are flat, it follows that the chain complexes $\mathbb{F}=\mathbb{F'}\otimes_{R_1}R $
and $\mathbb{G}=\mathbb{G'}\otimes_{R_2}R$ are free resolutions over $R$ of $M_1\otimes_{R_1}R$, and of  $M_2\otimes_{R_2}R$, respectively.
As before, we get that $\Tor_i^R(M_1\otimes_{R_1}R, M_2\otimes_{R_2}R) \cong H_i(\mathbb{F}\otimes_R \mathbb{G})$ for all $i$.

For any $R_1$-module $N_1$ and any $R_2$-module $N_2$ there exists a canonical isomorphism
$N_1\otimes_K N_2 \cong (N_1\otimes_{R_1}R)\otimes_R(N_2\otimes_{R_2}R)$, which can be used to construct an isomorphism between the chain complexes
$\mathbb{F}'\otimes_K \mathbb{G}'$ and $\mathbb{F}\otimes_R \mathbb{G}$. This implies that  $\mathbb{F}\otimes_R \mathbb{G}$ is also acyclic, which proves (i).

For (ii) we apply the identity in part (i) to the modules $R_1/\tr_{R_1}(M_1)$ and $R_2/\tr_{R_2}(M_2)$. Combined with \cite[Exercise 3.1.3]{Weibel},
that gives
 $$
0= \Tor_1^R\left(\frac{R}{(\tr_{R_1}(M_1))R}, \frac{R}{(\tr_{R_1}(M_2))R}\right) \cong \frac{(\tr_{R_1}(M_1))R \cap  (\tr_{R_1}(M_2))R}{ (\tr_{R_1}(M_1))R \cdot (\tr_{R_1}(M_2))R}.
$$
These equations together with Proposition \ref{prop:tensor-trace}, Lemma \ref{lemma:trace-changeofrings}(iii)
and the isomorphism of $R$-modules $M_1\otimes_K M_2 \cong (M_1\otimes_{R_1}R)\otimes_R (M_2\otimes_{R_2}R)$
yield the desired statement.
\end{proof}

\begin{Theorem}
\label{thm:canonical-tensor-product}
Let $R_1$ and $R_2$ be positively graded Cohen-Macaulay $K$-algebras over a field $K$, and $\omega_{R_1}$, $\omega_{R_2}$ their respective canonical modules.
Denote $R=R_1\otimes_K R_2$.
Then
$$
 \omega_R \cong \omega_{R_1} \otimes_K \omega_{R_2}  \text{ and }
$$
$$
\tr_R(\omega_R) = \tr_{R_1}(\omega_{R_1})R \cdot \tr_{R_2}(\omega_{R_2})R.
$$
\end{Theorem}

\begin{proof}
We consider minimal presentations  $R_1\cong S_1/I_1$, $R_2\cong S_2/I_2$ as quotients of the polynomial rings $S_1$ and $S_2$, and we let $S=S_1\otimes_K S_2$.
Clearly $R\cong S/(I_1, I_2)S$.

We note that the fibers of the flat extensions $S_1\subset S$ and $S_2\subset S$ are Gorenstein, hence using \cite[Theorem 3.3.14]{BH} we obtain that
$\omega_1=\omega_{R_1}\otimes_{S_1}S$ and $\omega_2=\omega_{R_2} \otimes_{S_2} S$ are
the canonical modules for $S_1/I_1\otimes_{S_1} S \cong S/I_1S$, and for $S/I_2S$, respectively.

Let
$$
\mathbb{F'}:  \quad 0\to F'_p \to F'_{p-1} \to \dots \to F'_{0} \to S_1/I_1 \to 0
$$
be a minimal free $S_1$-resolution of $S_1/I_1$.
Since the extension $S_1\subset S$ is flat, it follows that the complex  $\mathbb{F} = \mathbb{F'}  \otimes_{S_1}S$ is a minimal free resolution of $S/I_1S$ over $S$.
Then by \cite[Corollary 3.3.9]{BH}, the dual complex
$\mathbb{F}^* =\Hom_S(\mathbb{F} , S) \cong \Hom_{S_1}(\mathbb{F}' , S_1)\otimes_{S_1}S$ is a minimal free resolution of $\omega_1$.

Similarly, if we start with
$$
\mathbb{G'} :  \quad 0\to G'_q \to G'_{q-1} \to \dots \to G'_{0} \to S_2/I_2 \to 0
$$
a minimal free $S_2$-resolution of $S_2/I_2$, the complex $\mathbb{G} = \mathbb{G'}  \otimes_{S_2}S$ is a minimal free resolution of $S/I_2S$ over $S$, and
	$\mathbb{G}^* $ minimally resolves $\omega_2$ over $S$.

By \cite[Theorem 2.7.2]{Weibel},   $H_i(\mathbb{F} \otimes_S \mathbb{G})= \Tor^S_i(S/I_1S, S/I_2S)$ for all $i$.
We may apply Proposition \ref{prop:zero-tor} to the $S_1$-module $S_1/I_1$ and to the $S_2$-module $S_2/I_2$ and we obtain that
 the complex $\mathbb{F} \otimes_S \mathbb{G}$ is acyclic
and that it minimally resolves $S/I_1S \otimes_S S/I_2S \cong S/(I_1, I_2)S \cong R$ over $S$. Thus $\projdim_S R=\projdim_{S_1} R_1+ \projdim_{S_2}R_2$.

On the other hand, by the formula of Hoa and Tam in \cite[Theorem 1.3]{HoaTam} (see also \cite[Theorem 15.1(ii)]{Matsumura-CRT}) we have that $\dim R= \dim R_1 + \dim R_2$, hence $R$ is also Cohen-Macaulay.
This implies that $(\mathbb{F}  \otimes_S \mathbb{G} )^*$ is a minimal $S$-free resolution of $\omega_{R}$.

From the isomorphism  $(\mathbb{F} \otimes_S \mathbb{G} )^*\cong \mathbb{F}^* \otimes_S \mathbb{G}^*$ we derive that
$\mathbb{F}^* \otimes_S \mathbb{G}^*$ is acyclic, hence  it resolves (minimally) the $S$-module $\omega_1\otimes_S \omega_2$.
Since
$$
\omega_1\otimes_S \omega_2  \cong \omega_{R_1}\otimes_{S_1}S\otimes_S\omega_2\cong \omega_{R_1}\otimes_{S_1}(S_1\otimes_K S_2)\otimes_{S_2}\omega_{R_2}\cong \omega_{R_1}\otimes_K \omega_{R_2},
$$
we get that $\omega_R \cong \omega_{R_1} \otimes_K \omega_{R_2}$. Now using  Proposition \ref{prop:zero-tor} we   obtain that
$\tr_R(\omega_R) = \tr_{R_1}(\omega_{R_1})R \cdot \tr_{R_2}(\omega_{R_2})R$.
\end{proof}

\begin{Corollary}
\label{cor:tensor-NG}
Let $R_1$ and $R_2$ be positively graded Cohen-Macaulay $K$-algebras over a field $K$ with $\embdim R_i>0$ for $i=1,2$, and let $R=R_1\tensor_KR_2$. Then the following conditions are equivalent:
\begin{enumerate}
\item[{\em (i)}] $R$ is nearly Gorenstein;
\item[{\em (ii)}] $R$ is Gorenstein;
\item[{\em (iii)}] $R_1$ and $R_2$ are Gorenstein.
\end{enumerate}
\end{Corollary}

\begin{proof} (i)\implies (iii): Since $\tr(\omega_R)=\tr(\omega_{R_1})R\cdot \tr(\omega_{R_2})R$, our assumption implies that
$\mm\subseteq \tr(\omega_{R_1})R\cdot \tr(\omega_{R_2})R$. Suppose $ \tr(\omega_{R_i}) \subseteq \mm_{R_i}$ for some $i$, say $i=1$.
Then $\tr(\omega_R)\subseteq \mm_{R_1}R\subseteq \mm_R$. This is a  contradiction since the last inclusion is proper if $\embdim R_2>0$.
Thus  $ \tr(\omega_{R_i})=R_i$ for $i=1,2$, and this implies that $R_i$ is Gorenstein for $i=1,2$, see Lemma~\ref{non}.

(iii)\implies (ii): The condition implies that $\omega_{R_i}\iso R_i$ (up to a shift in the grading). Therefore, $\omega_R\iso R_1\tensor_K R_2$  (up to a shift in the grading). Hence $R$ is Gorenstein.

(ii) \implies (i) is obvious.
\end{proof}

\begin{Corollary}
\label{polynomial}
Let $R$ be a positively graded $K$-algebra and $S=R[x_1,\ldots,x_n]$ with $n\geq 1$ be the polynomial ring over $R$. Then $S$ is nearly Gorenstein if and only if  $R$ is Gorenstein.
\end{Corollary}

There is a local analogue to this corollary.

\begin{Proposition}
\label{localanalogue}
Let $R$ be a local Cohen-Macaulay ring admitting a canonical module, and let $S=R[|x_1,\ldots,x_n|]$ with $n\geq 1$ be the formal power series over $R$. Then $S$ is nearly Gorenstein if and only if  $R$ is Gorenstein.
\end{Proposition}

\begin{proof}
If $R$ is Gorenstein, then $S$ is Gorenstein and hence nearly Gorenstein. For the converse implication, observe that  $\omega_S=\omega_R\tensor_R S$.
Thus  we may apply Lemma~\ref{lemma:trace-changeofrings} and see that $\tr(\omega_S)=\tr(\omega_R)S$.
 If  $R$ is not Gorenstein, then $\tr(\omega_R)\subseteq \mm_R$, and hence $\tr(\omega_S)=\mm_RS\subsetneq \mm_S$. Therefore,  $S$ is not nearly Gorenstein.
\end{proof}

\subsection{Veronese subalgebras}

Let $R=\Dirsum_{i\in \ZZ} R_i$ be a standard  graded $K$-algebra over the field $K$ and $d$ a positive integer.
The ring $R^{(d)}=\oplus_{i\in \ZZ}R_{id}$ is called the  $d$-{th}  {\em Veronese subring} of $R$.
We will consider $R^{(d)}$ as a standard graded $K$-algebra by letting the elements (generators) in $R_d$ have degree $1$.

Let $M_j=\Dirsum_{i\in \NN}R_{di+j}$ for $j=0,\dots, d-1$. Then $R=\Dirsum_{j=0}^{d-1}M_j$, and $M_j$ is a finite $R^{(d)}$-module.
This explains why the $M_j$'s are called the $d$-th {\em Veronese submodules} of $R$.
We refer to \cite[Exercise 3.6.21]{BH} and  \cite{GW} for more background on Veronese algebras and their properties.

We denote $\mm$ and $\mm^{(d)}$ the maximal graded ideal of $R$, and of $R^{(d)}$, respectively.

In the sequel we show that  the trace of any $d$-th Veronese module of $R$ contains $\mm^{(d)}$.

\begin{Theorem}
\label{thm:vero-modules}
Let $R$ be a standard graded $K$-algebra  with $\depth (R) >0$  and $d$ a positive integer.
Then $\mm^{(d)} \subseteq \tr_{R^{(d)}}(M_j)$ for any $d$-th Veronese submodule $M_j$ of $R$.
\end{Theorem}

\begin{proof}
If $d=1$, there is only one Veronese submodule $M_0=R^{(d)}$ and its trace equals $R^{(d)}=R$.

Assume $d>1$.  Without loss of generality we may assume that the field $K$ is infinite, otherwise we extend the base field and we use Lemma \ref{lemma:trace-changeofrings}(iii).
Then we may choose a $K$-basis $x_1, \dots, x_n$ for  $R_1$ sucht that $x_1$ is a nonzero divisor.  
This basis generates the algebra $R$, as well.
Fix $0\leq j <d$. As an $R^{(d)}$-module, $M_j$ is generated by  the set $\mathcal{G}=\{ x_1^{a_1}\cdots  x_n^{a_n} : \sum_{i=1}^{n} a_i=j \}$.
Since $x_1$ is regular on $R$, the $R^{(d)}$-module  $I=x_1^{d-j} M_j$ is isomorphic to $M_j$, hence they have the same trace ideal.
Note that $I$ is an ideal of $R^{(d)}$ which contains the regular element $x_1^d$, and therefore the formula in Lemma \ref{lemma:traceideal} applies.

For a vector of nonnegative integers $\alpha=(a_1, \dots, a_n)$ we denote
 $|\alpha|=\sum_{i=1}^n a_i$ and  $x^\alpha=x_1^{a_1}\cdots     x_n^{a_n}$.
Given a  product $x^{\alpha}$ with $|\alpha|=d$, there exist vectors $\beta, \gamma$ with nonnegative integer entries such that $\alpha=\beta+\gamma$ and $|\beta|=j$.
We may write 
$$
x^\alpha= (x_1^{d-j}\cdot x^{\beta})\cdot \frac{x^{\gamma}\cdot x_1^j}{x_1^d},
$$
where $x^{\gamma}\cdot x_1^j \in R_d$.
Clearly, $x_1^{d-j}\cdot x^{\beta}$ is in $I$. As $x_1^d$ is regular on $R$ (and on $R^{(d)}$)  we get that
$\frac{x^\gamma \cdot x_1^j}{x_1^d} \in Q(R^{(d)})$.
Since $\frac{x^{\gamma}\cdot x_1^j}{x_1^d}(x_1^{d-j} \mathcal{G}) \subseteq R^{(d)}$, we conclude that $x^{\alpha} \in I\cdot I^{-1}$ and  $ \mm^{(d)} \subseteq \tr_{R^{(d)}}(M_j) $,
which finishes the proof.
\end{proof}

\begin{Corollary}
\label{cor:vero-ng}
If $R$ is a Gorenstein standard graded $K$-algebra   of positive dimension, then all its Veronese subalgebras are nearly Gorenstein.
\end{Corollary}

\begin{proof}
Goto and Watanabe in \cite{GW} showed that if $R$ is Gorenstein, then up to a shift $\omega_{R^{(d)}} \cong M_j$ for some $j \in 1,\dots, d-1$.
This fact  together with Theorem \ref{thm:vero-modules} proves the  statement.
\end{proof}

\begin{Corollary}
The Veronese subalgebras of a polynomial ring over any field  are nearly Gorenstein.
\end{Corollary}

\subsection{Squarefree Veronese subalgebras}

For $1\leq d \leq n$, the $d$-th squarefree Veronese subalgebra of $S=K[x_1, \dots, x_n]$ is the subalgebra
$R_{n,d}$ generated by the set $V_{n,d}$ of squarefree monomials in $S$ of degree $d$. As with the regular Veronese,
$R_{n,d}$ is standard graded by letting $\deg m=1$ for all $m\in V_{n,d}$. Moreover, $R_{n,d}$ is normal (cf. \cite{Villa}, and \cite{Sturmfels}),
hence Cohen-Macaulay.

In contrast to the regular Veronese subalgebras, we show that $R_{n,d}$ is not nearly Gorenstein unless it is already Gorenstein.
By work of De Negri and Hibi in \cite{DeNegri-Hibi} (see also \cite[Corollary 2.14]{BVV}) the latter
property   holds if and only if
$$
(n,d) \in \{(n,1), (n,n-1), (2d, d) \}.
$$

The monomials in $R_{n,d}$ correspond to lattice points in some rational cone $C\subset \RR^n$, and by Danilov's Theorem
the canonical module $\omega_{R_{n,d}}$ is the $K$-span of the monomials in the relative interior of $C$.
We refer to \cite[Chapter 6]{BH} for the necessary background. The next results of Bruns, Vasconcelos and Villarreal in \cite{BVV}
describe the lattice points of the cone $C$ and a generating system for  $\omega_{R_{n,d}}$.

\begin{Lemma} (\cite[Proposition 2.4]{BVV})
\label{lemma:sqfreevero}
A monomial $m=x_1^{a_1}\cdots x_n^{a_n}$ is in $R_{n,d}$ if and only if
\begin{eqnarray*}
\sum_{i=1}^n a_i \equiv 0 \mod d, \text{and} \\
d a_i \leq  \sum_{j=1}^{n}a_j, \text{ for all } i=1, \dots, n.
\end{eqnarray*}
\end{Lemma}

\begin{Theorem}
\label{thm:bvv-omega} (\cite[Theorem 2.6]{BVV})
If $n\geq 2d\geq 4$, the ideal $\omega_{R_{n,d}}$ is generated by the  monomials $m=x_1^{a_1}\cdots x_n^{a_n}$
which satisfy the conditions:
\begin{enumerate}
\item[{\em (i)}]  $a_i \geq 1$ and $d a_i \leq -1 + \sum_{j=1}^n a_j$, for all $i=1, \dots, n$;
\item[{\em (ii)}] $\sum_{i=1}^n a_i \equiv 0 \mod d$;
\item[{\em (iii)}] $|\{ i: a_i \geq 2 \} | \leq d-1$.
\end{enumerate}
\end{Theorem}

The restrictions on $n$ and $d$ in Theorem \ref{thm:bvv-omega} do not leave out any interesting cases.
If $d=1$, then $R_{n,d}=S$. As noted in \cite[Remark 2.13]{BVV}, the bijective map $\rho: V_{n,d} \to V_{n, n-d}$ given by
 $\rho(m)=x_1\cdots x_n/m$ for all $m$ in $V_{n,d}$ induces a graded isomorphism between the $K$-algebras $R_{n,d}$ and $R_{n,n-d}$,
so we may restrict our study to the case $n \geq 2d \geq 4$.

Our aim is to describe a generating set for the anti-canonical ideal of a squarefree Veronese algebra $R_{n,d}$.
The following lemma is a partial result in this direction.

\begin{Lemma}
\label{lemma:dvars}
 If $n>2d\geq 4$, then
$$
(1/x_{i_1}\cdots x_{i_d}\:\;  1\leq i_1 <\dots <i_d \leq n)R_{n,d} \subseteq \omega_{R_{n,d}}^{-1}.
$$
Equality holds if and only if $n=2d+1$.
\end{Lemma}

\begin{proof}
Let $f=1/x_{i_1}\cdots x_{i_d}$ with $ 1\leq i_1 <\dots <i_d \leq n$, and $m=x_1^{a_1}\cdots x_n^{a_n}$ a generator for $\omega_{R_{n,d}}$
as given by Theorem \ref{thm:bvv-omega}. Then $d a_i \leq -1+\sum_{j=1}^n a_j$ for all $i$, and since $d$ divides $\sum_{j=1}^n a_j$ we get
\begin{equation}
\label{eq:more-d}
d a_i \leq -d+\sum_{j=1}^n a_j, \text{ for } i=1,\dots, n.
\end{equation}
Since $a_i>0$ for all $i$, it follows that
$$
mf=\left(\prod_{j=1}^d x_{i_j}^{a_{i_j}-1}\right) \left(\prod_{i\notin\{i_1, \dots, i_d\}} x_i^{a_i}\right)
$$
is a monomial in $S$ of degree divisible by $d$.
The inequalities \eqref{eq:more-d} assure that the exponents of $mf$ satify the conditions in Lemma \ref{lemma:sqfreevero}, therefore $mf \in R_{n,d}$.
This confirms that $1/x_{i_1}\cdots x_{i_d} \in \omega_{R_{n,d}}^{-1}$.

We defer the discussion of the equality case after the proof of Theorem \ref{thm:sqfreevero-anticanonical}.
\end{proof}

For any monomial $u$ is $S$ we denote its vector of exponents by $\log u$.
Even if $u$ is in $R_{n,d}$, its degree will be  considered with respect to the standard grading on $S$ assigning to each variable the degree $1$.

\begin{Theorem}
\label{thm:sqfreevero-anticanonical}
Let $n>2d \geq  4$ . Then $\omega_{R_{n,d}}^{-1}$ is minimally generated by the fractions $u/x_1\cdots x_n$,  where  $u$ is a monomial in $S$ such that either
it is a product of $n-d$ distinct variables, or else,
 up to a permutation of the entries, $\log u$ belongs to  the set $P_{n,d}$, where $P_{n,d}$ is the set of vectors of the form
$$(\underbrace{c,\dots, c}_{d+1},b_{d+2}, \dots, b_{n-d-c+1}, \underbrace{0,\dots,0}_{d+c-1}) \in \NN^{n}$$  with
$ 2 \leq c \leq n-2d$, $c\geq b_{d+2} \geq \dots \geq b_{n-d-c+1} \geq 0$ and
$\sum_{i=d+2}^{n-d-c+1}b_i=n-2d-c$.
\end{Theorem}

\begin{proof}
Let $R=R_{n,d}$ and  $\mathcal{A}$ denote the set of generators of $\omega_R$ as given in Theorem~\ref{thm:bvv-omega}.

{\em Step 1.} Note that by permuting the exponents of the variables of any  monomial  in $\mathcal{A}$ we get another monomial in $\mathcal{A}$.
This observation together with  (i) and (iii) in Theorem \ref{thm:bvv-omega} assures that $x_1\cdots x_n$ is the greatest common divisor (computed in $S$) of the monomials in $\mathcal{A}$.

{\em Step 2.} Since $\omega_R$ is generated by monomials in $R$, it follows that $\omega_R^{-1}$ is also generated by monomials in $K[x_1^{\pm 1}, \dots, x_n^{\pm 1}]$.
If $f=u/v \in \omega_R^{-1}$ with $u$ and $v$ coprime polynomials in $S$,
then $v$ divides the greatest common divisor of the monomials in $\mathcal{A}$, which is $x_1\cdots x_n$.
This shows that  for the purpose of determining a generating set for $\omega_R^{-1}$, it is enough to consider
  fractions $f=u/(x_1\cdots x_n)$ where $u$ is in the set
\begin{equation*}
\mathcal{B}=\{ u \text{   monomial in } S:\;  \deg u \equiv n \mod d, \; u \cdot\omega_R \subseteq x_1\cdots x_n R \}.
\end{equation*}
A priori, for $u\in \mathcal{B}$ the fraction $u/(x_1\cdots x_n)$ is not necessarily in $Q(R)$.

{\em Step 3.} We identify the elements of  $\mathcal{B}$ as follows.
A monomial $u=x_1^{b_1}\cdots x_n^{b_n}$ is in $\mathcal{B}$ if and only if
\begin{equation*}
x_1^{a_1+b_1-1}\cdots x_n^{a_n+b_n-1} \in R \text{ for all } x_1^{a_1}\cdots x_n^{a_n} \in \mathcal{A}.
\end{equation*}

This is equivalent, by Lemma \ref{lemma:sqfreevero}, to the fact that
\begin{eqnarray}
\label{eq:u-in-B}
\deg u &\equiv& n \mod d, \quad \text{and} \\
db_i+ (da_i+n-d) &\leq & \deg u+ \sum_{j=1}^n a_j, \text{ for } 1\leq i \leq n, \text{ and } x_1^{a_1}\cdots x_n^{a_n} \in \mathcal{A}.\nonumber
\end{eqnarray}
Letting
$$
E=\max\{ da_i+n-d-\sum_{j=1}^n a_j: 1\leq i \leq n,\;  x_1^{a_1}\cdots x_n^{a_n} \in \mathcal{A}\},
$$
the condition \eqref{eq:u-in-B} is equivalent to
\begin{eqnarray}
\label{eq:u-in-B-with-E}
\deg u \equiv n \mod d, \text{ and }
db_i+ E \leq \deg u,  \text{ for } 1\leq i \leq n.
\end{eqnarray}

 We note that since the monomials in $\mathcal{A}$ are closed to permutation of the exponents, the value of $E$ does not depend on $i$.
It immediately follows from \eqref{eq:u-in-B-with-E} that if $u'$ is obtained from $u \in \mathcal{B}$ by permuting its exponents, then also $u'\in \mathcal{B}$.

{\em Step 4.}
We claim that
$$E=n-2d.$$
It is easy to check with Theorem \ref{thm:bvv-omega}
that
$$
(\underbrace{n-2d+1, \dots, n-2d+1}_{d-1}, \underbrace{1,\dots,1}_{n-d+1})
$$
is the exponent of a monomial in $\mathcal{A}$, hence $E\geq n-2d$.

On the other hand, $1/(x_1\cdots x_d) \in \omega_R^{-1}$ by Lemma \ref{lemma:dvars}, hence $x_{d+1}\cdots x_n \in \mathcal{B}$.
From \eqref{eq:u-in-B-with-E} we get that $d+E\leq n-d$, which shows that $E=n-2d$.
Therefore \eqref{eq:u-in-B-with-E} becomes
\begin{eqnarray}
\label{eq:u-in-B-with-E-new}
\deg u \equiv n \mod d, \text{ and }
db_i+ (n-2d) \leq \deg u,  \text{ for } 1\leq i \leq n.
\end{eqnarray}

{\em Step 5.} A consequence of \eqref{eq:u-in-B-with-E} is that  there are at least $d+1$ distinct variables  dividing $u$.
Indeed, arguing by contradiction, if there are at most $d$ distinct variables dividing $u$ and say, $b_{d+1}=\dots =b_n=0$,
we get by summing the inequality in \eqref{eq:u-in-B-with-E} for $i=1,\dots, d$ that
$$
\sum_{i=1}^d (db_i +E) \leq d\deg u= d \sum_{i=1}^db_i,
$$
 equivalently
$dE\leq 0$, which is false.

{\em Step 6.}
We define a partial order on $\mathcal{B}$ by letting $u \leq v$ if $v/u \in R$.
Our next goal is to identify the elements of $\mathcal{B}$  minimal  with respect to this partial order.

For $u=x_1^{b_1}\cdots x_n^{b_n}$  in $\mathcal{B}$ we set
$$
\mathcal{I}(u)=\{x_i: d b_i+(n-2d)=\deg u \}.
$$

It follows from \eqref{eq:u-in-B-with-E-new}, arguing modulo $d$, that if for some $i$ we have $d b_i+ E<\deg u$, then  $d b_i+ E \leq \deg u -d$.

We claim that
\begin{equation}
\label{eq:i}
u \text{ is minimal in } \mathcal{B} \iff |\mathcal{I}(u)| \geq d+1.
\end{equation}

Indeed, if $|\mathcal{I}(u)| \leq d$, then  there exists $m$ a squarefree monomial of degree $d$ which contains the variables in $\mathcal{I}(u)$.
Then  $u'=u/m$ verifies \eqref{eq:u-in-B-with-E-new}, hence $u' \in \mathcal{B}$ and $u$ is not minimal in $\mathcal{B}$.

On the other hand, if $|\mathcal{I}(u)| \geq d+1$, then no matter by which product of $d$ distinct variables we divide $u$ to obtain a monomial $u'$,
for at least some $x_{i_0} \in \mathcal{I}(u)$ the inequality in \eqref{eq:u-in-B-with-E-new} applied to $u'$ and $i_0$ does not hold.
This confirms \eqref{eq:i}.

Based on \eqref{eq:i}, we describe the minimal elements in $\mathcal{B}$ as follows.
From \eqref{eq:u-in-B-with-E-new} we infer that   $\deg u \geq n-d$,  and equality holds if and only if
$u$ is a product of $n-d$ distinct variables.  These monomials are clearly minimal in $\mathcal{B}$.

Assume $u$ is minimal in $\mathcal{B}$ and $\deg u > n-d$. Since $\deg u \equiv n \mod d$ we have $\deg u \geq n$.
We may pick $i_1<\dots<i_{d+1}$ in $\mathcal{I}(u)$ and we let $c=b_{i_1}=\dots= b_{i_{d+1}}$.
Clearly, $c=\max\{b_i:1\leq i \leq n\}$ is greater than $1$, otherwise $u=x_1\cdots x_n$ for which $\mathcal{I}(u)=\emptyset$, a contradiction.
Then
$$
n-2d=c+ \sum_{j\notin\{i_1, \dots, i_{d+1}\}}b_j.
$$
It follows that $2\leq c \leq n-2d$ and that for at least $(n-d-1)-(n-2d-c)=d+c-1$ indices $j \notin \{i_1, \dots, i_{d+1}\}$ we have $b_j=0$.
Equivalently, up to a permutation of the entries, $\log u$ is in $P_{n,d}$.

{\em Step 7.} We claim that for any $u$ a minimal element in $\mathcal{B}$ we have that $f=u/(x_1\cdots x_n)$ is in $Q(R)$.
That would imply that these $f$'s generate $\omega_R^{-1}$ minimally.

We write $n=\ell d+r$ with $1\leq r\leq d$ and we let $m$  be the product of some distinct $r$ variables dividing $u$.
Since $n>2d$, we get $n\geq 2d+r$, hence by \eqref{eq:u-in-B-with-E-new} we obtain
$$
db_i \leq \deg u - (n-2d)  \leq \deg u -r, \text{ for } 1\leq i \leq n.
$$
This implies that $u/m$ is in $R$, using Lemma \ref{lemma:sqfreevero}.
Clearly $(x_1\cdots x_n)/m$ is also in $R$ as a product of  $\ell d$ distinct variables, hence $f\in Q(R)$.
This concludes the proof.
\end{proof}

We can now finish the proof of Lemma \ref{lemma:dvars}.

\begin{proof} (of Lemma \ref{lemma:dvars}, continued)
According to Theorem \ref{thm:sqfreevero-anticanonical},
the equality holds in Lemma \ref{lemma:dvars} if and only if the set $P_{n,d}$   is empty.

As noted in the proof of Theorem \ref{thm:sqfreevero-anticanonical}, any vector in $P_{n,d}$ has at least $d+1$ non zero entries,
and at least $d+1$ zero entries, hence $n\geq  2d+2$. This shows that $P_{n,d}=\emptyset$ for $n=2d+1$.
Conversely, if $n\geq 2d+2$, it is easy to check that
$$
(\underbrace{n-2d, \dots, n-2d}_{d+1}, \underbrace{0, \dots, 0}_{n-d-1}) \in P_{n,d},
$$
hence $P_{n,d}$ is not empty in this case.
\end{proof}

\begin{Example}
{\em We computed the sets $P_{n,2}$ for small  values of $n$.
For brevity, we record only the non zero entries of the vectors in $P_{n,2}$, as is it clear how many zeroes one should add to reconstruct $\log u$.\\
$P_{5,2}= \emptyset$, $P_{6,2}=\{(2,2,2)\}$, $P_{7,2}=\{(2,2,2,1), (3,3,3) \}$,\\
$P_{8,2}=\{(2,2,2,2), (2,2,2,1,1), (3,3,3,1),  (4,4,4)\}$,\\
$P_{9,2}=\{  (2,2,2,1,1,1)  , (2,2,2,2,1),  (3,3,3,1,1),  (3,3,3,2), (4,4,4,1),(5,5,5)\}$,\\
$P_{10,2}=\{(2,2,2,2,2), (2,2,2,2,1,1), (2,2,2,1,1,1,1), (3,3,3,3), (3,3,3,2,1), \\
(3,3,3,1,1,1), (4,4,4,2), (4,4,4,1,1), (5,5,5,1), (6,6,6) \}$. %\\
%$P_{11,2}=\{(2,2,2, 1,1,1,1,1),  (2,2,2,2,1,1,1), (2,2,2,2,2,1), (3,3,3,1,1,1,1), \\
%(3,3,3,2,1,1), (3,3,3,2,2), (3,3,3,3,1), (4,4,4,1,1,1), (4,4,4,2,1), (4,4,4,3), \\
%(5,5,5,1,1),   (5,5,5,2),  (6,6,6,1), (7,7,7) \}$.
}
\end{Example}

The next result shows that a squarefree Veronese subalgebra of $S$ is not nearly Gorenstein, unless it is already Gorenstein.

\begin{Theorem}
\label{thm:sqfreevero} For  $1\leq d \leq n$ let $R=R_{n,d}$. The following are equivalent:
\begin{enumerate}
\item[{\em (i)}] $R$ is nearly Gorenstein;
\item[{\em (ii)}] $R$ is Gorenstein;
\item[{\em (iii)}] $d=1$ or $d=n-1$ or $n=2d$.
\end{enumerate}
\end{Theorem}

\begin{proof}
It is clear that (ii) implies (i). The equivalence between (ii) and (iii) was shown by De Negri and Hibi in \cite{DeNegri-Hibi}, and also by Bruns, Vasconcelos and Villarreal in \cite{BVV}.
So, it is enough to prove that if $2< 2d <n$, then $\tr (\omega_R)\subsetneq \mm_R$.
Indeed, if $m$ is a monomial generator for $\omega_R$ as given by Theorem \ref{thm:bvv-omega},
and  $f=u/(x_1\cdots x_n)$ is a generator for $\omega_R^{-1}$ as given by Theorem \ref{thm:sqfreevero-anticanonical}, then $\deg m \geq n$ and $\deg u \geq n-d$.
Thus, $mf$ is a monomial of degree at least $n-d\geq d+1$ and $\tr (\omega_R)\subsetneq \mm_R$.
\end{proof}

\subsection{Segre products}

If $R$ and $S$ are $\NN$-graded $K$-algebras, their Segre product $T=R\sharp S$ is the graded $K$-algebra whose components are
$T_i=R_i \otimes_K S_i$ for  all $i$.

In this part we assume that $R$ and $S$ are standard graded Gorenstein rings.
We further need that $T$ is Cohen--Macaulay. For example, if one assumes that  $\dim R,\dim S\geq 2$, 
then $T$ is Cohen--Macaulay  if and only if $a(R),a(S)<0$,  see  \cite[Theorem 4.4.4]{GW}. 
Here $a(R'):=-\min\{i\: (\omega_{R'})_i\neq 0\}$ denotes the $a$-invariant of a graded Cohen--Macaulay $K$-algebra $R'$.

We recall that an affine semigroup $H$ is a finitely generated submonoid of $\NN^d$. 
Moreover, we say that $H$ is homogeneous  if all its minimal generators lie on an affine hyperplane. 
Equivalently, the semigroup ring $K[H]$ is standard graded  by assigning the degree one 
to all the monomials corresponding to the minimal generators of $H$. In that case, we also say that $K[H]$ is a homogeneous semigroup ring.

If the semigroup rings $R=K[H_1]$ and $S=K[H_2]$ are homogeneous, then it is easy to see that their Segre product $R\sharp S$ is again a homogeneous semigroup ring.

\begin{Theorem}
\label{general segre}
Let $R$ and $S$ be standard graded Gorenstein $K$-algebras  of Krull  dimension at least $2$ with $a$-invariants $a(R)=r$ and $a(S)=s$,
and  assume that $T=R\sharp S$ is Cohen--Macaulay.
Then the following hold:
\begin{enumerate}
\item[{\em (i)}] $\tr(\omega_T)\supseteq\mm_T^{|r-s|}$;
\item[{\em (ii)}] if $R$ and $S$ are homogeneous semigroup rings, then $\tr(\omega_T)=\mm_T^{|r-s|}$.
\end{enumerate}
\end{Theorem}

\begin{proof}
(i) By \cite[Theorem 4.3.1]{GW} of Goto and Watanabe one has that
$\omega_T \cong \omega_R\sharp \omega_S$.
Since $\omega_R \cong R(r)$ and $\omega_S\cong S(s)$, it follows that we may identify $\omega_T$ with $R(r)\# S(s)$. We may assume $s\geq r$.
Then it can be seen that $\omega_T$ is generated as a $T$-module by $(\omega_T)_{-r}=R_0\tensor_K S_{s-r}$. Each  $f\in R_{s-r}$ induces for each $i$ a $K$-linear map
\[
(\omega_T)_i=R_{i+r}\tensor S_{i+s}\to R_{i+s}\tensor S_{i+s}=T_{i+s}, \quad a\tensor b\mapsto fa\tensor b.
\]
These $K$-linear maps  compose to a $T$-linear map $f\:\omega_T\to T$.
The image of this map is   generated by  $f\tensor S_{s-r}$.
Thus, as $f$ varies over all $f\in R_{s-r}$ we see that  $\tr(\omega_T)\supseteq R_{s-r}\tensor S_{ s-r}$.
This yields the desired conclusion, since  $\mm_T^{s-r}$ is generated by the elements of  $R_{s-r}\tensor S_{ s-r}$.

(ii)  We choose a non-zero  monomial $f\in R_{s-r}$.
Since  clearly, $T$  and $R$ are  domains, the $T$-module homomorphism $f\: \omega_T\to T$, defined in (i)  is injective.
Therefore, $\omega_T$ is isomorphic to its image in $T$, and hence due to (i)
it suffices to show that $I^{-1}\cdot I\subseteq\mm_T^{s-r}$ for the ideal $I$ in $T$ generated by $f\tensor S_{s-r}$.

Since $I$ is    generated by monomials, its inverse $I^{-1}$ is   generated by monomials , as well. 
Let $x\in I^{-1}$, where
$
x=(g_1\tensor g_2)/(h_1\tensor h_2)=(g_1/h_1)\tensor (g_2/h_2)
$
with $g_1\tensor g_2\in T_a$ and $h_1\tensor h_2\in T_b$.
Then $fg_1/h_1\tensor g_2/h_2S_{s-r}\subseteq T$.
In particular, $g_2/h_2\in J^{-1}$ where   $J=\mm_S^{s-  r}$.
Since $\dim S\geq 2$, it follows that  $\grade J\geq 2$, so that $J^{-1}=S$.  Therefore,  $a-b\geq 0$,  which implies that $xI\subseteq \mm_T^{s-r}$.
\end{proof}

\begin{Corollary}
\label{nearlysegre}
Assume  the hypotheses of Theorem \ref{general segre}.

If $|r-s|\leq 1$ then the Segre product $T=R\# S$ is nearly Gorenstein.
Moreover, if     $R$ and $S$ are homogeneous semigroup rings,  and
 $T$ is nearly Gorenstein, then 
 $|r-s|\leq 1$.
\end{Corollary}

%\begin{Remark}
%{\em
%If $R\otimes_K S$ is a domain, then $R\sharp S$ is also a domain  being a subalgebra of $R\otimes_K S$.
%It is known from \cite[Chapter 5, \S 17]{Bourbaki}  that if the field $K$ is algebraically closed and $R$ and $S$
%are  domains, then so is $R\otimes_K S$.
%}
%\end{Remark}

% By  \cite[Remark (4.0.3)(v)]{GW}, if $K$ is algebraically closed and $R$ and $S$ are normal domains, then so is $R\sharp S$.

In the special case that $R=K[x_1,\ldots,x_r]$ and $S=K[y_1,\ldots,y_s]$ are polynomial  rings and $r\geq s\geq 2$, the anti-canonical ideal of $T$ can be easily computed. We remark that $T$ may be identified with
$$
K[x_iy_j:1\leq i\leq r, 1\leq j \leq s] \subset K[x, y].
$$
Any monomial $u\in K[x_1, \dots, x_r, y_1, \dots, y_s]$ may be uniquely written $u=u_x u_y$, where $u_x$ and $u_y$ are monomials in $R$, and respectively in $S$. 
After this identification, a $K$-basis of $T$ is given by the monomials $u=u_xu_y$ with $\deg u_x=\deg u_y$. 
Furthermore, by Theorem~\ref{general segre}, since $a(R)=-r$ and $a(S)=-s$,  the canonical module of $T$ is isomorphic to the ideal
\begin{equation*}
\label{eq:omega-segre}
I=(x_1^{r-s} y^\beta\:\; |\beta|=r-s) \subset T.
\end{equation*}

\begin{Proposition}
\label{lemma:i}
With notation introduced, $I^{-1}$ is generated as a $T$-module by the fractions $x^{\alpha}/x_1^{r-s}$ with $0\leq \alpha$ and $|\alpha|=r-s$.
\end{Proposition}

\begin{proof}
Let $f=x^{\alpha}/x_1^{r-s}$ with $0\leq \alpha$ and $|\alpha|=r-s$.
Then $f=(x^{\alpha}y_1^{r-s})/(x_1y_1)^{r-s}$ is in $Q(T)$ and for any generator of $I$,  $g=x_1^{r-s}y^\beta$ with $|\beta|=r-s$, clearly $fg=x^\alpha y^\beta$ is in $T$.
Therefore, the  given set of fractions is part of $I^{-1}$. We show that it  generates it, too.

Let $u, v$ be coprime polynomials in $K[x,y]$  with $u/v \in I^{-1}$. Then $v$ divides $x_1^{r-s}y^\beta$ for all $\beta \geq 0$. Since $s\geq 2$ we get that $v$ divides $x_1^{r-s}$.
As $T$ is a monomial subalgebra of $K[x,y]$, it follows that for determining a generating system of $I^{-1}$
it is enough to consider fractions of the form $f=u/x_1^{r-s}$ with $u$ a monomial in $K[x,y]$. Such an $f$  belongs to $Q(T)$ if and only if  $\deg u_x=(r-s)+\deg u_y$.
If that is the case, we may decompose $u_x=x^\alpha \cdot x^{\alpha'}$ with $\alpha, \alpha'$ vectors with nonnegative integer entries and $|\alpha|=r-s$.
Letting $w=x^{\alpha'}u_y$, since  $|\alpha'|=\deg u_y$ we get that $w \in T$.
Therefore $f=(x^\alpha/x_1^{r-s})w$, which shows that the given set of fractions generates $I^{-1}$. This completes the proof of the proposition.
\end{proof}

\begin{Remark}
{\em
Proposition~\ref{lemma:i} implies that $\omega_T^{-1}$ is isomorphic to the ideal
\[
J=( x^\alpha y_1^{r-s}\:\; |\alpha|=r-s) \subset T.
\]
}
\end{Remark}

\medskip

\section{Nearly Gorenstein Hibi rings}
\label{sec:hibi}
As an application of the results of the previous section on Segre products we will classify in this section the nearly Gorenstein Hibi rings.

Let $P = \{p_1, \ldots, p_n \}$ be a finite partially ordered set
(``poset'' for short).
A {\em chain} of $P$ is a totally ordered subset of $P$.
The {\em length} of a chain $C$ is $|C| - 1$.
The {\em rank} of $P$ is the maximal length of chains of $P$
and it is denoted by $\rank(P)$.  A poset $P$ is called {\em pure}
if the length of each maximal chain of $P$ is equal to $\rank(P)$.

A {\em poset ideal} of $P$ is a subset $\alpha \subset P$
with the property that if $a \in \alpha$ and $b \in P$ with $b \leq a$,
then $b \in \alpha$.  In particular, the empty set, as well as $P$ itself,
is a poset ideal.  Let $\Jc(P)$ denote the set of poset ideals of $P$.
If $\alpha$ and $\beta$ are poset ideals of $P$, then each of
$\alpha \cap \beta$ and $\alpha \cup \beta$ is again a poset ideal of $P$.
Hence $\Jc(P)$ is a finite lattice (\cite[p.~157]{HH}) ordered by inclusion.

A finite lattice $L$ is called {\em distributive} if, for all $a, b, c$
belonging to $L$, one has
\begin{eqnarray*}
(a \vee b) \wedge c =  (a \wedge c) \vee (b \wedge c), \\
(a \wedge b) \vee c =  (a \vee c) \wedge (b \vee c).
\end{eqnarray*}
For example, for an arbitrary finite poset $P$, the finite lattice
$\Jc(P)$ is distributive.  Furthermore,
Birkhoff's fundamental structure theorem
for finite distributive lattices (\cite[Theorem 9.1.7]{HH})
guarantees the converse.  More precisely, given a finite distributive
lattice $L$, there is a unique finite poset $P$ with $L = \Jc(P)$.

Let $S = K[x_1, \ldots, x_n, s]$ denote the polynomial ring in $n + 1$ variables
over a field $K$.  Given a poset ideal $\alpha \in \Jc(P)$,
we introduce the squarefree monomial
\[
u_\alpha = \left(\prod_{p_i \in \alpha}x_i \right)s
\]
which belongs to $S$.  In particular,
$u_\emptyset = s$ and $u_P = x_1 \cdots x_n s$.
In \cite{TH87} the toric ring
\[
\mathcal{R}_K[L] = K[\{ \, u_\alpha \, : \, \alpha \in \Jc(P) \, \}] \subset S
\]
is introduced   and it is nowadays called {\em the Hibi ring of $L$ defined over $K$}.
 It is standard graded by letting $\deg(u_\alpha)=1$ for each $\alpha \in \Jc(P)$.
 The toric ring $\mathcal{R}_K[L]$ is normal and Cohen--Macaulay  of dimension $|P|+1$, see \cite{TH87}.
It follows that the quotient field of $\mathcal{R}_K[L]$
is the rational function field $K(x_1, \ldots, x_n, s)$.
Furthermore, the $a$-invariant $a(\mathcal{R}_K[L])$ of $\mathcal{R}_K[L]$
coincides with $-(\rank(P)+2)$.

Let $\widehat{P} = P \cup \{-\infty, +\infty \}$ with
$-\infty < a < +\infty$ for each $a \in P$.
Write $\Omega(P)$ for the set of those order-reversing maps
$\delta : \widehat{P} \to \ZZ_{\geq 0}$ with $\delta(+\infty) = 0$
and $\Omega^*(P)$ for the set of those strictly order-reversing
maps $\delta : \widehat{P} \to \ZZ_{\geq 0}$ with $\delta(+\infty) = 0$.
Let $\Mc(P)$ denote the set of those maps
$\gamma : \widehat{P} \to \ZZ$ with $\gamma(+\infty) = 0$
for which $\delta + \gamma \in \Omega(P)$ for each $\delta \in \Omega^*(P)$.

Given a map $\xi : \widehat{P} \to \ZZ_{\geq 0}$ with $\xi(+\infty) = 0$,
we introduce the monomial
\[
u^\xi = \left(\prod_{i=1}^{n}x_i^{\xi(p_i)}\right)s^{\xi(-\infty)}
\]
which belongs to $S$.  It is shown in \cite{TH87} that $\mathcal{R}_K[L]$
with $L = \Jc(P)$
is spanned by those monomials $u^\delta$ with $\delta \in \Omega(P)$.
Let $\omega_{\mathcal{R}_K[L]}$ denote the ideal of $\mathcal{R}_K[L]$
which is generated by those monomials $u^\delta$ with $\delta \in \Omega^*(P)$.
Then the ideal $\omega_{\mathcal{R}_K[L]}$ is isomorphic to the canonical ideal
of $\mathcal{R}_K[L]$.  It then follows that $\mathcal{R}_K[L]$
is Gorenstein if and only if $P$ is pure.

We now turn to the problem of finding a characterization on $P$
for which the toric ring $\mathcal{R}_K[L]$ with $L = \Jc(P)$
is nearly Gorenstein.
Given $a \in P$, we introduce the intervals
$[-\infty, a] = \{ \, b \in \widehat{P} \, : \, -\infty \leq b \leq a \, \}$
and
$[a, +\infty] = \{ \, b \in \widehat{P} \, : \, a \leq b \leq +\infty \, \}$
of $\widehat{P}$.

\begin{Lemma}
\label{interval}
Let $P$ be a finite poset and suppose that the toric ring $\mathcal{R}_K[L]$
of the distributive lattice $L = \Jc(P)$ is nearly Gorenstein.
Then, for each element $a \in P$, the intervals $[-\infty, a]$
and $[a, +\infty]$ of $\widehat{P}$ are pure.
\end{Lemma}

\begin{proof}
Let $a \in P$ for which $[-\infty, a]$ is nonpure
and define $\rho \in \Omega(P)$ by setting $\rho(z) = 1$
if $z \in [-\infty, a]$ and $\rho(z) = 0$
if $z \not\in [-\infty, a]$.
Since $\mathcal{R}_K[L]$ is nearly Gorenstein,
there exist $\delta \in \Omega^*(P)$ and $\gamma \in \Mc(P)$
with $\rho = \delta + \gamma$.
Since $[-\infty, a]$ is nonpure, there exist $y$ and $y'$ belonging to
$[-\infty, a]$ for which $y'$ covers $y$ with $\delta(y) > \delta(y') + 1$.
Since $y'$ covers $y$, it follows from \cite[Corollary (2.8)]{TH88} that
there is $\delta' \in \Omega^*(P)$ with $\delta'(y) = \delta'(y') + 1$.
Furthermore, since $\gamma \in \Mc(P)$, one has
$\delta' + \gamma \in \Omega(P)$.
Now, since $(\delta + \gamma)(y) = (\delta + \gamma)(y') = 1$,
it follows that $(\delta' + \gamma)(y) < (\delta' + \gamma)(y')$,
which contradicts $\delta' + \gamma \in \Omega(P)$.

Let $a \in P$ for which $[a, +\infty]$ is nonpure
and define $\rho \in \Omega(P)$ by setting $\rho(z) = 1$
if $z < a$ and $\rho(z) = 0$
if $z \not< a$.
Again, since $\mathcal{R}_K[L]$ is nearly Gorenstein,
there exist $\delta \in \Omega^*(P)$ and $\gamma \in \Mc(P)$
with $\rho = \delta + \gamma$.
Since $[a, +\infty]$ is nonpure, there exist $y$ and $y'$ belonging to
$[a, +\infty]$ for which $y'$ covers $y$ with $\delta(y) > \delta(y') + 1$.
Since $y'$ covers $y$, again by using \cite[Corollary (2.8)]{TH88},
one has $\delta' \in \Omega^*(P)$ with $\delta'(y) = \delta'(y') + 1$.
Furthermore, since $\gamma \in \Mc(P)$, one has
$\delta' + \gamma \in \Omega(P)$.
Now, since $(\delta + \gamma)(y) = (\delta + \gamma)(y') = 0$,
it follows that $(\delta' + \gamma)(y) < (\delta' + \gamma)(y')$,
which contradicts $\delta' + \gamma \in \Omega(P)$.
\end{proof}

\begin{Lemma}
\label{pure}
Let $P$ be a finite connected poset
and suppose that, for each element $a \in P$,
the intervals $[-\infty, a]$ and $[a, +\infty]$ of $\widehat{P}$
are pure.  Then $P$ is pure.
\end{Lemma}

\begin{proof}
Let $a_1, \ldots, a_r$ denote the maximal elements of $P$.
Then $(-\infty, a_1], \ldots, (-\infty, a_r]$
are pure, where
$(-\infty, a_i] = \{ \, b \in P \, : \, b \leq a_i \, \}$.
If $(-\infty, a_i] \cap (-\infty, a_j] \neq \emptyset$,
then $\rank((-\infty, a_i]) = \rank((-\infty, a_j])$.
In fact, if $a \in (-\infty, a_i] \cap (-\infty, a_j]$,
then the fact that $[-\infty, a]$ and $[a, +\infty]$ are pure
guarantees $\rank((-\infty, a_i]) = \rank((-\infty, a_j])$.
Since $P$ is connected, after rearranging $a_1, \ldots, a_s$,
it follows that
$(\cup_{i=1}^{j} (-\infty, a_i]) \cap (-\infty, a_{j+1}] \neq \emptyset$
for $1 \leq j < r$.  Hence
$\rank((-\infty, a_1]) = \cdots =  \rank((-\infty, a_r])$.
Since every maximal chain of $P$ belongs to one of
$(-\infty, a_1], \ldots, (-\infty, a_r]$,
it follows that $P$ is pure, as desired.
\end{proof}

Combined,  Lemmata \ref{interval} and \ref{pure} guarantee the following.

\begin{Corollary}
\label{result}
Let $P$ be a finite poset and suppose that the toric ring $\mathcal{R}_K[L]$
of the distributive lattice $L = \Jc(P)$ is nearly Gorenstein.
Then every connected component of $P$ is pure.
\end{Corollary}

In general, if a poset  $P$ is the disjoint union of the finite posets
$P_1, \ldots, P_q$, then $\mathcal{R}_K[L]$ is the Segre product
$\mathcal{R}_K[L_1] \sharp \mathcal{R}_K[L_2] \sharp \cdots \sharp \mathcal{R}_K[L_q]$,
where $L = \Jc(P)$ and $L_i = \Jc(P_i)$ for $1 \leq i \leq q$.

\begin{Theorem}
\label{nearlygorhibiring}
Let $P$ be a finite poset. Then the toric ring $\mathcal{R}_K[L]$
of the distributive lattice $L = \Jc(P)$ is nearly Gorenstein
if and only if $P$ is the disjoint union of pure connected posets
$P_1, \ldots, P_q$ such that $|\rank(P_i) - \rank(P_j)| \leq 1$
for $1 \leq i < j \leq q$.
\end{Theorem}

\begin{proof}
Let $P$ be the disjoint union of connected posets
$P_1, \ldots, P_q$ and $L_i = \Jc(P_i)$ for $1 \leq i \leq q$.
Then $\mathcal{R}_K[L] = \mathcal{R}_K[L_1] \sharp \cdots \sharp \mathcal{R}_K[L_q]$.
Recall that $\mathcal{R}_K[L_i]$ is Gorenstein if and only if $P_i$ is pure.
Furthermore, $a(\mathcal{R}_K[L_i]) = -(\rank(P_i)+2)$.
Hence the ``if'' part follows from Theorem \ref{general segre}.
Now, Corollary \ref{result} says that if
$\mathcal{R}_K[L]$ is nearly Gorenstein, then each of $P_1, \ldots, P_q$
is pure.  Thus the ``only if'' part
% Since $a(\mathcal{R}_K[L_i]) = -(\rank(P_i)+2)$
follows again from Theorem \ref{general segre}.
\end{proof}

\medskip

\section{The one-dimensional case  and the almost Gorenstein property}
\label{sec:onedimensional}
Let $(R,\mm, K)$ be a  local Cohen--Macaulay ring with canonical module $\omega_R$.
The results and proofs of this section are equally valid for  positively   graded $K$-algebras.

Barucci and Fr\"oberg \cite{BF} introduced almost Gorenstein rings for one-dimensional  local rings which are analytically unramified. We use here the description of almost Gorenstein rings  due to  Goto, Takahashi and  Taniguchi \cite{GTT} which allows an extension of this concept to higher dimensions.
Goto et al.  \cite{GTT} call $R$ {\em almost Gorenstein} if there exists an exact sequence
\begin{eqnarray}
\label{almost}
0\to R\to \omega_R\to C\to 0
\end{eqnarray}
of $R$-modules such that $\mu(C)=e(C)$, where $e(C)$ denotes the multiplicity of $C$ with respect to $\mm$. One also says that $C$ is an Ulrich module.

 In the case that $R$ is $1$-dimensional, the Krull dimension of $C$ is equal to $0$. Since an Ulrich module of dimension zero is annihilated by the maximal ideal,  for $C$  to be an Ulrich module it means that $\mm C=0$.

It is of interest to compare the nearly Gorenstein and  the almost Gorenstein properties.
We have the following implication.

\begin{Proposition}
\label{implication}
Let $R$ be an $1$-dimensional almost Gorenstein ring. Then $R$ is nearly Gorenstein.
\end{Proposition}

\begin{proof}
The exact sequence (\ref{almost}) yields the exact sequence
\[
0\to \Hom_R(\omega_R, R)\to R\to \Ext_R^1(C,R)
\]
Let $f\in \omega_R$ be the image of $1\in R$ under the map $R\to \omega_R$. Then  the map $ \Hom_R(\omega_R, R)\to R$ is defined by assigning to each $\varphi\in  \Hom_R(\omega_R, R)$ the element $\varphi(f)\in R$. Thus $\tr(\omega_R)$ contains the image of $\Hom_R(\omega_R, R)$ in $R$. Since $\mm$ annihilates $\Ext_R^1(C,R)$,  this image contains $\mm$, and the desired result follows.
\end{proof}

\begin{Remark}
\label{example-ng-not-ag}
{\em
 In general,  the class of $1$-dimensional nearly Gorenstein rings is much larger than that of $1$-dimensional almost Gorenstein rings.  A simple example of   nearly Gorenstein ring,  but which is not almost Gorenstein  is the subring $R=K[|t^5,t^6,t^7|]$ of the formal power series ring $K[|t|]$. That $R$ is indeed nearly Gorenstein can be seen from the next result.
In a forthcoming paper we explicitly describe  the nearly Gorenstein semigroup rings $K[|H|]\subset K[|t|]$,  where $H$ is any $3$-generated numerical semigroup. 
}
\end{Remark}

\begin{Proposition}
\label{codim2}
Let $(S,\nn)$ be a $3$-dimensional regular local ring, $I\subset  \nn^2$ an ideal generated by $3$ elements such that  $R=S/I$ is a 1-dimensional 
Cohen-Macaulay  ring. 
Let $A$ be the relation matrix of $I$. Then $R$ is nearly Gorenstein if and only if $I_1(A)=\nn$.
\end{Proposition}

\begin{proof}
The assumptions imply that $I$ has the resolution
\begin{eqnarray}
\label{short}
0\to S^2\to S^3 \to I\to 0.
\end{eqnarray}
In particular, the type of $R$ is 2. Now we apply Corollary~\ref{typetwo}.
\end{proof}

Coming back to the example in Remark~\ref{example-ng-not-ag},  we write $R= K[|t^5,t^6,t^7|]$ as a factor ring of $S=K[|x_1,x_2,x_3|]$ by considering the $K$-algebra homomorphism $\epsilon\: S\to R$ with $\epsilon(x_1)=t^5$, $\epsilon(x_2)=t^6$, $\epsilon(x_3)=t^6$. The kernel $I$ of $\epsilon$ is generated by $$x_1^4-x_2x_3^2, x_2^2-x_1x_3, x_3^3-x_1^3x_2.$$ 
The relation matrix $A$  of this ideal is
\[
A= \left( \begin{array}{ccc} x_1 & x_2 & x_3^2\\
x_2 & x_3 & x_1^3 \end{array}\right),
\]
thus $I_1(A)=(x_1,x_2,x_3)$. Using Proposition \ref{codim2} we get that  $R$ is   nearly Gorenstein.

\begin{Proposition}
\label{type2}
 Assume in addition to the setup of Proposition~\ref{codim2} that $R$ is generically Gorenstein. Then    $\omega_R$ is an ideal in $R$ and  $\mu(\omega_R^{-1})   \leq 3$.
\end{Proposition}

\begin{proof} We apply Corollary \ref{third}. 
\end{proof}

\begin{Remark}
\label{higher}
{\em
 In contrast to the $1$-dimensional case, it does not follow in higher dimensions that an almost Gorenstein ring is nearly Gorenstein.
Indeed, if $R$ is almost Gorenstein, then the formal power series ring $R[|x|]$ is again almost Gorenstein, see \cite[Theorem 3.7]{GTT}.
On the other hand, by Proposition~\ref{localanalogue}, if $R$ is not Gorenstein, then $R[|x|]$ is not nearly Gorenstein.
}
\end{Remark}

A celebrated inequality of Abhyankar gives that $\embdim (R)\leq e(R) +\dim(R)-1$ for any Cohen-Macaulay local ring with an  infinite residue field. When  equality holds, one says that $R$ has minimal multiplicity.

One of the referees remarked that under this extra assumption, a nearly Gorenstein local ring is also almost Gorenstein. We reproduce his proof below. 

\begin{Theorem}
\label{thm:minimal}
Let $(R,\mm)$ be a Cohen-Macaulay local ring with $\dim R >0$ and infinite residue field. Assume that $R$ possesses the canonical module $\omega_R$. If $R$ is nearly Gorenstein and it has minimal multiplicity, then it is an almost Gorenstein local ring.
\end{Theorem}

\begin{proof}
We may assume that $R$ is not a Gorenstein ring. 
We first remark that it is enough to prove the case when $\dim R=1$. Indeed, let us assume that the  conclusion holds in that case. Let $a_1,\dots, a_d$ be a regular sequence in $\mm$, where $d=\dim R$. By Proposition \ref{behaviour} the $1$-dimensional Cohen-Macaulay local  ring $R'=R/(a_1,\dots,a_{d-1})$ is nearly Gorenstein and of minimal multiplicity. Then, by our assumption we get that $R'$ is almost Gorenstein. Using \cite[Theorem 3.7]{GTT} we obtain that $R$ is an almost Gorenstein local ring, too.

Let us assume that $\dim R=1$. Since $R$ is nearly Gorenstein, it is generically Gorenstein and $\omega_R$ can be identified with an ideal of $R$, see \cite[Proposition 3.3.18]{BH}. As $R/\mm$ is an infinite field, by \cite[Corollary 2.9]{GMT} there exists an $R$-submodule $K$ of $Q(R)$ such that $R\subseteq K \subseteq \overline{R}$ and $K\cong \omega_R$ as $R$-modules, where $\overline{R}$ denotes the integral closure of $R$ in $Q(R)$.

We set $S=R[K]=\union_{\ell \geq 0} K^\ell$ and $\aa=R:K$. Here and in the rest of the proof the colon ideals are taken over $Q(R)$.
%There exists $x\in \mm$  regular element such that $K=\frac{1}{x}I$ and $I$ is an ideal in $R$. Then $I$ is also a canonical ideal for $R$. 
Since $R$ is nearly Gorenstein but not Gorenstein 
% $\mm=\tr(\omega_R)=\tr(I)=I\cdot (R:I)=(xK) \cdot (R:xK)=(xK):(\frac{1}{x}(R:K))=K\cdot (R:K)=\aa K$. 
 $$\mm=\tr(\omega_R)= \tr(K)=\aa K.$$
Then 
$$
\aa S= \aa \cdot \Union_{\ell \geq 0} K^\ell=\Union_{\ell \geq 0} \aa K^\ell=\Union_{\ell \geq 1}\aa K^\ell=\Union_{\ell\geq 0}\mm K^\ell=\mm S.
$$
Consequently, since $R\subseteq S$ is an integral extension of rings, by \cite[Proposition 1.6.1]{HunekeSwanson} we obtain
$$
\overline{\aa} =  \overline{\aa S}   \cap R   =\overline{\mm S} \cap  R= {\mm},
$$
where $\overline{ \phantom{a}}$ denotes the integral closure of an ideal. 

As $\mm \subseteq \overline{\aa}$, from \cite[Corollary 1.2.5]{HunekeSwanson} we derive that $\aa$ is a reduction of $\mm$.
Since $\aa$ is an $\mm$-primary ideal in $R$ which is a $1$-dimensional  Cohen-Macaulay ring, any minimal reduction if $\aa$ is a principal ideal, see \cite[Corollary 1, pp. 154]{NR}. We choose an element $f \in \aa$ such that $(f)$ is a minimal reduction of $\aa$.
It follows from \cite[Proposition 1.2.4]{HunekeSwanson}  that $(f)$ is a minimal reduction of $\mm$, too. 

Set $B=\union_{\ell \geq 1} (\mm^\ell: \mm^\ell)$. Because $R$ has minimal multiplicity, $\mm$ is a stable ideal (\cite[Corollary 1.10]{Lipman}), i.e.   $\mm^\ell: \mm^\ell= \mm:\mm$ for all $\ell >0$.
 Now \cite[Lemma 1.11]{Lipman} implies that $\mm^2=f\mm$ and $f$ is regular. Since $f\mm \subseteq \aa \mm \subseteq \mm^2$ we get that $\mm^2=\aa \mm=\aa^2 K$.

Subtituting $\mm=\aa K$ into $\mm^2=\aa \mm$ we obtain that  $\aa^2K^2=\aa^2 K$, hence $\mm^2=\aa^2 K^\ell$ for all $\ell >0$. This implies that $\aa^2 S= \mm^2$.

Since $S$ is a ring,
$$
S \subseteq \aa^2S: \aa^2 S =\mm^2: \mm^2=B=\mm:\mm \subseteq R:\mm.
$$
As $K\subseteq S$, we obtain that $\mm K \subset R$. This shows by definition that $R$ is an almost Gorenstein local ring.
\end{proof}

\medskip
{\bf Acknowledgement}.
We gratefully acknowledge the use of  the Singular (\cite{Sing}) software and of the numericalsgps package (\cite{Num-semigroup})  in GAP (\cite{GAP}) for our computations.
 We thank the anonymous referees for their comments which improved the exposition and strengthened some results, in particular for kindly suggesting Theorem  \ref{thm:minimal} and its proof. 

Dumitru Stamate  was supported by a fellowship at the Research Institute of the University of Bucharest (ICUB).
{}
\end{document}